\documentclass[11pt,dvips,twoside,letterpaper]{article}
\usepackage{bbm}

\usepackage{pslatex}
\usepackage{fancyhdr}
\usepackage{graphicx}
\usepackage{geometry}
\usepackage[latin1]{inputenc}
\usepackage{mathrsfs}
\usepackage{amsmath}
\usepackage{amssymb}
\usepackage{amsmath}
\usepackage{amsfonts}
\usepackage{graphicx}
\usepackage{amsmath}
\usepackage[english]{babel}
\usepackage{cite}

\RequirePackage[latin1]{inputenc}
 \RequirePackage[T1]{fontenc}

\def\figurename{Figure} 
\makeatletter
\renewcommand{\fnum@figure}[1]{\figurename~\thefigure.}
\makeatother

\def\tablename{Table} 
\makeatletter
\renewcommand{\fnum@table}[1]{\tablename~\thetable.}
\makeatother

\usepackage{color}
\usepackage{amsmath}
\usepackage{amssymb}
\usepackage{amsfonts}
\usepackage{amsthm,amscd}

\newtheorem{theorem}{Theorem}[section]
\newtheorem{lemma}[theorem]{Lemma}

\theoremstyle{definition}
\newtheorem{definition}[theorem]{Definition}

\theoremstyle{remark}
\newtheorem{remark}[theorem]{Remark}

\numberwithin{equation}{section}

\begin{document}

\title{Stability equivalence for stochastic differential
equations, stochastic differential delay equations and their
corresponding Euler-Maruyama methods in $G$-framework
 \thanks{The work is supported
by the National Natural Science Foundation of China (11871076) and
the Natural Science Foundation of Shandong Province (ZR2017MA015).}}

\author{Wen Lu \footnote{E-mail: llcxw@163.com}
\\
   School of Mathematics and Information Sciences,
 \\ \small Yantai University, Yantai 264005, China  }
\date{}
\maketitle
\begin{abstract} In this paper, we investigate the stability equivalence problem for stochastic
differential delay equations, the auxiliary stochastic differential
equations and their corresponding Euler-Maruyama (EM) methods under
$G$-framework. More precisely, for $p\geq 2$, we prove the
equivalence of practical exponential stability in $p$-th moment
sense among stochastic differential delay equations driven by
$G$-Brownian motion ($G$-SDDEs), the auxiliary stochastic
differential equations driven by $G$-Brownian motion ($G$-SDEs), and
their corresponding Euler-Maruyama methods, provided the delay or
the step size is small enough. Thus, we can carry out careful
simulations to examine the practical exponential stability of the
underlying $G$-SDDE or $G$-SDE under some reasonable assumptions.
\end{abstract}

\vspace{.08in} \noindent \textbf{Keywords}  $G$-Brownian motion;
$G$-stochastic differential equation; $G$-stochastic differential
delay equations; Euler-Maruyama method; practical exponential
stability; equivalence

\vspace{.08in} \noindent \textbf{Mathematics Subject Classification}
 60H10; 60H35

\section{Introduction}
It is acknowledged that stochastic differential equations (SDEs) and
stochastic differential delay equations (SDDEs) are widely used in
many branches of science and engineering fields. As the vital
important characteristic of the dynamical systems, the stability of
SDEs or SDDEs has been studied intensively for the theories and
applications, one can see Hasminskii \cite{Hasm12}, Mao
\cite{Mao07}, Mao and Yuan \cite{MaoYuan06} and references therein
for more details.

Motivated by the potential applications in uncertainty problems,
risk measures and the super-hedging in finance, a nonlinear
framework named $G$-framework was established by Peng \cite{Peng07}.
Under this framework, Peng \cite{Peng07, Peng10} introduced the
$G$-expectation, $G$-Brownian motion and set up the associated
stochastic calculus of It\^{o}'s type. As a counterpart in the
classical Brownian motion framework, stochastic differential
equations driven by $G$-Brownian motion ($G$-SDEs) were firstly
introduced by  Peng \cite{Peng10}. Under the global Lipschitz
assumption, Peng \cite{Peng10} and Gao \cite{Gao09} have proved the
wellposedness of $G$-SDEs.  Comparing with the classical SDE, the
$G$-SDEs has wide applications owing to the model uncertainty. Since
then, many works have been done on $G$-expectation, $G$-Brownian
motion and the stability of the related $G$-SDEs, see Denis et al.
\cite{DHP11}, Hu and Peng \cite {HuPeng09}, Hu et al.
\cite{Hu14,Hu20}, Li et al. \cite{LLL16}, Li and Peng \cite {LiP11},
Ren \cite{ren17, ren18},  Soner et al. \cite{Soner11}, Song
\cite{Song11, Song12}, Zhang and Chen \cite{ZhCh12} and the
references therein.

Taking into account the time-delay influence in $G$-framework, Fei
et al. \cite{FFY19} firstly studied the existence and uniqueness and
the second kind of stability of the solutions to stochastic
differential delay equations driven by $G$-Brownian motion
($G$-SDDEs) under the local Lipschitz and the linear growth
conditions. Fei et al. \cite{FFMY22} investigated the
delay-dependent asymptotic stability for a class of highly nonlinear
$G$-SDDEs, Yao and Zong \cite{YZ20} discussed delay-dependent
stability of a class of stochastic delay systems driven by
$G$-Brownian motion. Besides, Zhu and Huang \cite{ZhuHuang20}
studied the $p$-th moment exponential stability for a class of
stochastic delay nonlinear systems by $G$-Brownian motion  by using
the stochastic delay feedback controls.

Since most SDEs or SDDEs cannot be solved explicitly due to its
nonlinearity, the numerical solution is introduced naturally. In the
classical framework, the numerical methods on stochastic
differential equations has been well established and we refer the
reader to, for example, \cite{Higham02}, \cite{Hu96},
\cite{Kloeden92},\cite{LMY19},
 \cite{Mao07}, \cite{Yuan06}and
references therein. Similarly, in the case of $G$-framework, for the
numerical scheme  of $G$-SDEs, see Li and Yang \cite{LiYang18}, Liu
and Lu \cite{LiuLu24}, Ullah and Faizullah \cite{RF17}, Yuan
\cite{Yhy21} and references therein for more details. It is well
known that the most powerful tool to judge the stability of the
dynamical systems is the so-called Lyapunov second method, which is
firstly established by Lyapunov in 1892. Using this method, one can
judge the stability of systems without knowledge of the solution of
the systems explicitly, and then this type stability of systems is
called the Lyapunov stability.

However, finding an appropriate Lyapunov function is not an easy
task in applications, therefore numerical methods become important.
This arises a natural problem: in the absence of Lyapunov function,
can we judge the stability of the systems and the corresponding
numerical methods from each other? In the classical framework, there
are a great deal of literatures to answer this problem positively.
For example, Higham et al. \cite{Higham03} initialized the study on
the equivalence of stabilities between the exact solutions of SDEs
and their numerical solutions, the authors in \cite{Higham03} proved
that underlying solutions are mean square exponential stable if and
only if numerical solutions are also exponential stable in the mean
square sense. Mao \cite{Mao2007} proved the similar equivalence
result to SDDEs under a global Lipschitz assumption. Furthermore,
Mao \cite{Mao2015} generalized the results in Higham et al.
\cite{Higham03} to the case of $p$-th moments. Liu et al.
\cite{LLD2018} investigated the equivalence of the mean square
exponential stability between the neutral delayed stochastic
differential equations and the Euler-Maruyama numerical scheme. More
recently, under the assumptions of convergence and boundedness on
the numerical methods, Bao et al. \cite{Bao21} established the
equivalence of $p$-th ($p>0$) moment stability between SDDEs and
their numerical methods. Very recently, under a global Lipschitz
assumption, Zhang et al. \cite{zhsongLiu23} investigated the
equivalence of the exponential stability in $p$-th moment sense
among the SDDEs, the auxiliary SDEs and their corresponding
Euler-Maruyama numerical scheme. In $G$-framework, Deng et al.
\cite{DFFM19} extended results in Mao \cite{Mao2007} to SDDEs driven
by $G$-Brownian motion. Yang and Li \cite{LiYang19} discussed the
equivalence of the $p$-th ($p\in (0, 1)$) moment exponential
stability between $G$-SDEs and the stochastic $\theta$-method.

As we know, a significant characteristic of the Lyapunov stability
is that the origin is an equilibrium point. However, there are many
systems where the origin is not an equilibrium point. In the case
that the origin is not an equilibrium point, it is still possible to
analyze the stability of the systems in a small neighborhood of the
origin, therefore the practical stability is introduced, see Sucec
\cite{Sucec87},  Lakshmikantham et al. \cite{LLM90}. Meanwhile,
owing to extensive applications in the qualitative behavior and
quantitative properties of systems, the practical stability has
attracted the attention of many researchers, seeCaraballo et al.
\cite{Caraballo14, {Caraballo17}}, Dlala and Hammami\cite{Dlala07},
Feng and Liu\cite{FengLiu94}, Yao et al. \cite{YLWW20}, Zhao
\cite{zhao08} and references therein. Moreover, in $G$-framework,
e.g., Zhu \cite{Zhu22} investigated the practical exponential
stability of a class of $G$-SDDEs via vector $G$-Lyapunov function.
Zhu et al. \cite{ZYL22} studied the global practical uniform
exponential stability for impulsive $G$-SDDEs.
 Caraballo et al. \cite{Caraballo23} obtained the practical stability with respect to
a part of the variables of $G$-SDEs.

We mention that, all the works about the stability equivalence
mentioned above were devoted to the exponential stability in the
mean square or $p$-th moment sense. To our best knowledge, there are
few works on the equivalence of the practical stability, especially
in the case of $G$-framework, it seems that Yuan and Zhu
\cite{YuanZhu24} is the only one. In \cite{YuanZhu24}, the authors
established the practical mean square stability equivalence between
the discrete and the continuous EM approximations for $G$-SDDE. In
addition to that, they also proved that the continuous EM
approximation shares the practical exponential stability of $G$-SDDE
in mean square sense.

Motivated by the previous literature, especially enlightened by the
recent work \cite{zhsongLiu23} due to Zhang et al., in this paper,
we shall study the equivalence of the practical stability among
$G$-SDDEs, the auxiliary $G$-SDEs and their EM methods in $p$-th
moment sense. More precisely,  we prove that, for any $p\geq 2$, the
practical exponential stability in $p$-th moment sense among
 $G$-SDDEs, $G$-SDEs, and their corresponding
EM methods can be deduced from each other, provided the time delay
or the step size is small enough. With such a theoretical result, in
the absence of Lyapunov function, the practical exponential
stability of $G$-SDDEs or $G$-SDEs can be affirmed by the numerical
simulation in practice. The rest of the paper is organized as
follows. Section 2 presents some notations and preliminaries for the
whole work. In Section 3, we establish the boundedness of the exact
and numerical solution for  $G$-SDDEs and $G$-SDEs. In Section 4, we
will present our main results, namely we will prove the equivalence
of practical exponential stability in $p$-th moment sense among
$G$-SDDEs, $G$-SDEs, and their corresponding EM methods. We conclude
the paper in the last section.

 \section{Notations and preliminaries}
Throughout this paper, unless otherwise specified, we use the
following notations.
 Let $\mathbb{R}=(-\infty, +\infty)$, $\mathbb{R}^{+}=[0, +\infty)$,  $a^{+}=\max\{0, a\}$ and $a^{-}=-\min\{0, a\}$ for
$a\in\mathbb{R}$. Let $|\cdot|$ be the Euclidean norm and
$\langle\cdot,\cdot\rangle$ be the scalar product in
$\mathbb{R}^{n}$. If $A$ is a vector or matrix, its transpose is
denoted by $A^{T}$. If $A$ is a matrix, its trace norm is denoted by
$|A|=\sqrt{\operatorname{trace}\left(A^{T} A\right)}$. If $x$ is a
real number, its integer part is denoted by $\lfloor x\rfloor$.

Let $\Omega$ be a given nonempty fundamental space and $\mathcal{H}$
a linear space of real functions defined on $\Omega$ such that (1)
$1 \in \mathcal{H} ;(2) \mathcal{H}$ is stable with respect to
bounded Lipschitz functions, i.e., for all $n \geq 1, X_1, \cdots,
X_n \in \mathcal{H}$ and $\varphi \in C_{b, lip
}\left(\mathbb{R}^n\right)$, it holds also $\varphi\left(X_1,
\cdots, X_n\right) \in \mathcal{H}$, where $C_{b, li
p}\left(\mathbb{R}^n\right)$ is the space of linear function
$\varphi: \mathbb{R}^n \rightarrow \mathbb{R}$, for $\forall x, y
\in \mathbb{R}^n$
$$
C_{b, lip }\left(\mathbb{R}^n\right):=\left\{\varphi \mid \exists C
\in \mathbb{R}^{+}, m \in N \text { s.t. }|\varphi(x)-\varphi(y)|
\leq C\left(1+|x|^m+|y|^m\right)|x-y|\right\}.
$$

\begin{definition}
A sublinear expectation $\mathbb{E}(\cdot)$ on $\mathcal{H}$ is a
functional $\mathbb{E}(\cdot): \mathcal{H} \rightarrow \mathbb{R}$
with the following properties: for each $X, Y \in \mathcal{H}$, we
have
\\\indent(1) \textbf{Monotonicity}: if $X \geq Y$, then $\mathbb{E}(X)
\geq \mathbb{E}(Y)$;
\\\indent(2) \textbf{Preservation of constants}:
$\mathbb{E}(c)=c$, for all $c \in \mathbb{R}$;
\\\indent(3)
\textbf{Sub-additivity}: $\mathbb{E}(X)-\mathbb{E}(Y) \leq
\mathbb{E}(X-Y)$;
\\\indent(4) \textbf{Positive homogeneity}:
$\mathbb{E}(\lambda X)=\lambda \mathbb{E}(X)$, for all $\lambda \in
\mathbb{R}^{+}$.
\end{definition}
The triple $(\Omega, \mathcal{H}, \mathbb{E})$ is called a sublinear
expectation space.
\begin{definition}In a sublinear expectation space $(\Omega,
\mathcal{H}, \mathbb{E})$, an $n$-dimensional random vector $Y=(Y_1,
. ., Y_n) \in \mathcal{H}^n$ is said to be independent from an
$m$-dimensional random vector $X=(X_1, . ., X_m) \in \mathcal{H}^m$
under the sublinear expectation $\mathbb{E}$, if for any test
function $\varphi \in C_{b, lip}\left(\mathbb{R}^{m+n}\right)$
$$
\mathbb{E}(\varphi(X,
Y))=\mathbb{E}\left(\left.\mathbb{E}(\varphi(x,
Y))\right|_{x=X}\right) .
$$
\end{definition}
\begin{definition} Let $X_1$ and $X_2$ be two $n$-dimensional random
vectors defined on sublinear expectation spaces $\left(\Omega_1,
\mathcal{H}_1, \mathbb{E}_1\right)$ and $\left(\Omega_2,
\mathcal{H}_2, \mathbb{E}_2\right)$, respectively. They are called
identically distributed, denoted by $X_1 \stackrel{d}{=} X_2$, if
$$
\mathbb{E}_1(\varphi(X_1))=\mathbb{E}_2(\varphi(X_2)), \quad \forall
\varphi \in C_{b, lip}(\mathbb{R}^n).
$$
$\bar{X}$ is said to be an independent copy of $X$ if $\bar{X}
\stackrel{d}{=} X$ and $\bar{X}$ is independent from $X$.
\end{definition}

\begin{definition} A $d$-dimensional random vector $X$ in a sublinear
expectation space $(\Omega, \mathcal{H}, \mathbb{E})$ is called
$G$-normal distributed if for each $\varphi \in C_{b,
lip}(\mathbb{R}^d)$,
$$
u(t, x):=\mathbb{E}(\varphi(x+\sqrt{t} X)), \quad t \in
\mathbb{R}^{+}, \quad x \in \mathbb{R}^d,
$$
is the viscosity solution to the following PDE defined on
$\mathbb{R}^{+} \times \mathbb{R}^d$ :
$$
\left\{\begin{array}{l}
\frac{\partial u}{\partial t}-G\left(D^2 u\right)=0 \\
\left.u\right|_{t=0}=\varphi,
\end{array}\right.
$$
where $G=G_X(A): \mathbb{S}^d \rightarrow \mathbb{R}$ is defined by
$$
G_X(A):=\frac{1}{2} \mathbb{E}((A X, X)),
$$
$D^2 u=\left(\partial_{x_i x_i}^2 u\right)_{i, j=1}^d$ and
$\mathbb{S}^d$ denoted the space of $d\times d$ symmetric matrices.
\end{definition}

Let $\Omega$ be the space of all $\mathbb{R}^d$-valued continuous
paths $\left(\omega_t\right)_{t \geq 0}$ with $\omega_0=0$. We
assume moreover that $\Omega$ is a metric space equipped with the
following distance:
$$
\rho\left(\omega^1, \omega^2\right):=\sum_{N=1}^{\infty}
2^{-N}\left(\max _{0 \leq t \leq
N}\left(\left|\omega_t^1-\omega_t^2\right|\right) \wedge 1\right),
$$
and consider the canonical process $B_t(\omega)=\omega_t, t \in
\mathbb{R}^{+}$, for $\omega \in \Omega$. For each $T\geq 0$, set
$$
L_{ip}^0\left(\Omega_T\right):=\{\varphi\left(B_{t_1}, B_{t_2},
\ldots, B_{t_n}\right): n \geq 1,0 \leq t_1 \leq \ldots \leq t_n
\leq T, \varphi \in C_{b, lip}(\mathbb{R}^{d \times n})\},
$$
where $C_{b, lip}(\mathbb{R}^{d \times n})$ denotes the space of
bounded and Lipschitz continuous functions on $\mathbb{R}^{d \times
n}$. Let $p \geq 1$,  for $T\in\mathbb{R}^+$, we define
$$
M_G^{p, 0}([0, T]):=\left\{\eta_t=\sum_{k=0}^{N-1} \xi_k
\mathbf{1}_{\left[t_k, t_{k+1}\right)}(t): \forall N\in \mathbb{N},
0=t_0<\cdots<t_N=T, \xi_k \in L_{i
p}^0\left(\Omega_{t_k}\right)\right\}
$$
and we denote by $M_G^p([0, T])$ the completion of $M_G^{p, 0}([0,
T])$ under the norm
$$\|\eta\|_{M_G^p([0, T])}:=\left(\mathbb{E}\int_0^T
\left|\eta_t\right|^p d t\right)^{\frac{1}{p}}.
$$

Now, let's introduce the definition of $G$-Brownian motion and
$G$-expectation.
\begin{definition} Let $(\Omega,
L_{ip}^0(\Omega_T), \mathbb{E})$ be a sublinear expectation space,
the sublinear expectation $\mathbb{E}(\cdot)$ is called
$G$-expectation if the canonical process $\left(B_t\right)_{t \geq
0}$ is  a
$G$-Brownian motion under $\mathbb{E}(\cdot)$ , that is: \\
(i) $B_0=0$; \\
(ii) for $t, s \geq 0$, the increment $B_{t+s}-B_t
\stackrel{\text{d}}{=} \sqrt{s} X$, where $X$ is $G$-normal
distributed;
\\ (iii) for $t, s
\geq 0$, the increment $B_{t+s}-B_t$ is independent from
$\left(B_{t_1}, B_{t_2}, \ldots, B_{t_n}\right)$ for each $n \in
\mathbb{N}$, and $0 \leq t_1 \leq t_2 \leq \ldots \leq t_n \leq t$.
\end{definition}
Let $\mathcal{H}_{t}$ be a filtration generated by $G$-Brownian
motion $(B(t))_{t \geq 0}$. Throughout this paper, , if not stated
otherwise, we set $p \geq 2$ , $\tau>0$ and denote by $C([-\tau, 0]
; \mathbb{R}^{n})$ the family of all continuous
$\mathbb{R}^{n}$-valued functions $\xi$ defined on $[-\tau, 0]$ with
the norm $\|\xi\|:=\sup _{-\tau \leq \theta \leq 0}|\xi(\theta)|$.
Let $L_{\mathcal{H}_{0}}^{p}\left([-\tau, 0] ;
\mathbb{R}^{n}\right)$ be the family of all
$\mathcal{H}_{0}$-measurable, $C([-\tau, 0]; \mathbb{R}^{n})$-valued
random variable $\xi=\{\xi(r): -\tau\leq r \leq 0\}$, such that
$\left\|\xi\right\|_{\mathbb{E}}^{p}:=\sup _{-\tau \leq r \leq 0}
\mathbb{E}|\xi(r)|^{p}<\infty$.
 If $x(t)$ is a continuous
$\mathbb{R}^{n}$-valued stochastic process on $t \in[-\tau,
\infty)$, we let $x_{t}:=\{x(t+\theta):-\tau \leq \theta \leq 0\}$
which is regarded as a $C\left([-\tau, 0] ;
\mathbb{R}^{n}\right)$-valued stochastic process.

In this article, we consider the following $G$-SDDE
 \begin{eqnarray}\label{eq:1}
d x(t)=f(x(t), x(t-\tau)) d t+g(x(t), x(t-\tau))d\langle B\rangle(t)
+h(x(t), x(t-\tau)) d B(t), \quad t \geq 0
\end{eqnarray}
where $f, g, h: \mathbb{R}^{n} \times \mathbb{R}^{n} \rightarrow
\mathbb{R}^{n}$, as well as $f, g, h \in M_{G}^{p}\left([-\tau, T] ;
\mathbb{R}^{n}\right), \forall T \geq 0$, and $B(t)$ be a
one-dimensional $G$-Brownian motion under the $G$-expectation space
$(\Omega, L_{ip}^0(\Omega_T), \mathbb{E}(\cdot))$ with
$G(a):=\frac{1}{2} \mathbb{E}\left(a
B(1)^{2}\right)=\frac{1}{2}\left({\bar\sigma}^{2}
a^{+}-{\underline\sigma}^{2} a^{-}\right)$, for $a \in \mathbb{R}$,
where ${\bar\sigma}^{2}=\mathbb{E}\left(B(1)^{2}\right)$,
${\underline\sigma}^{2}=-\mathbb{E}\left(-B(1)^{2}\right)$, $0\leq
\underline\sigma \leq\bar\sigma<\infty$. The quadratic variation
process of $G$-Brownian motion $B(t)$  is denoted by $\langle B
\rangle (t)$. The initial data $x_0=\xi=\{\xi(\theta):-\tau \leq
\theta \leq 0\} \in L_{\mathcal{H}_{0}}^{p}\left([-\tau, 0] ;
\mathbb{R}^{n}\right) $.
 For $G$-SDDE \eqref{eq:1}, we impose the
following assumptions:

\textbf{(H1)} The coefficients $f(\cdot, \cdot), g(\cdot, \cdot),
h(\cdot, \cdot): \mathbb{R}^{n} \times \mathbb{R}^{n} \rightarrow
\mathbb{R}^{n}$ are Borel measurable functions and satisfy global
Lipschitz condition, i.e., for each $x, x^{\prime},y, y^{\prime}\in
\mathbb{R}^{n}$, there exists a positive constant $L$ such that
$$
\left|f(x,y)-f(x^{\prime},y^{\prime})\right|^2\vee |
g(x,y)-g(x^{\prime},y^{\prime})|^2\vee |
h(x,y)-h(x^{\prime},y^{\prime})|^2 \leq
L(|x-x^{\prime}|^2+|y-y^{\prime}|^2).
$$
For the purpose of stability study in this paper, we assume that
$f(0,0)\neq 0$ or $h(0,0)\neq 0$ or $g(0,0)\neq 0$, i.e., the
$G$-SDDE \eqref{eq:1} does not exist the trivial solution $x\equiv
0$. In addition,  (H1) means the following linear growth condition
\begin{eqnarray}\label{lgc:2}
\left|f(x,y)\right|^2\vee | g(x,y)|^2\vee | h(x,y)|^2 \leq
\widehat{L}(1+|x|^2+|y|^2), \forall x, y\in \mathbb{R}^{n},
\end{eqnarray}
where $\widehat{L}=2\max\{L,
|f(0,0)|^2\vee|g(0,0)|^2\vee|h(0,0)|^2\}$. Under  (H1), the $G$-SDDE
\eqref{eq:1} admits a unique solution  on $T\geq -\tau$, see
\cite{FFY19}, we denote the solution by $x(t; 0, \xi)$. Similar to
the result in the classical Brownian motion framework, the solutions
to the $G$-SDDE \eqref{eq:1} have the following flow property:
$$x(t; 0, \xi)=x(t; s, x_s),\ \forall 0\leq s<t<\infty.$$
The auxiliary (nondelay) $G$-SDE for the $G$-SDDE \eqref{eq:1} is
\begin{eqnarray}\label{eq:2}
dy(t)=f(y(t), y(t)) dt+g(y(t), y(t))d\langle B\rangle(t) +h(y(t),
y(t))dB(t), \ t \geq 0,
\end{eqnarray}
where we define the initial value $y(0)=x(0)=\xi(0)$.

Let's now define the Euler-Maruyama (EM) approximate solutions to
the $G$-SDDE \eqref{eq:1} and its auxiliary  $G$-SDE \eqref{eq:2},
respectively. Given a step size $\Delta=\tau/m$ for a positive
integer $m$, let $t_n=n \Delta$ for $n \geq-m$. We set
$X_n=\xi(n\Delta), \forall -m\leq n \leq 0$ and $Y_0=\xi(0)$, then
the discrete EM solutions to the $G$-SDDE \eqref{eq:1} and the
$G$-SDE \eqref{eq:2} are defined by
\begin{eqnarray}\label{eq:d1}
X_{n+1}=X_{n}+f(X_{n}, X_{n-m})\Delta+g(X_{n}, X_{n-m})
\Delta\langle B\rangle_n+h(X_{n}, X_{n-m})\Delta B_n , \ n\geq 0
\end{eqnarray}
and
\begin{eqnarray}\label{eq:d2}
Y_{n+1}=Y_{n}+f(Y_{n}, Y_{n})\Delta+g(Y_{n}, Y_{n})\Delta\langle
B\rangle_n+h(Y_{n}, Y_{n})\Delta B_n, \ n\geq 0,
\end{eqnarray}
respectively, where $\Delta\langle B\rangle_n=\langle
B\rangle(t_{n+1})-\langle B\rangle(t_{n})$  and $\Delta
B_n=B(t_{n+1})-B(t_n)$.

Define $X_{\Delta}(t)=X_n$, for $t \in\left[t_n, t_{n+1}\right)$
with the initial value $X_{\Delta}(t)=\xi(t)$ on $[-\tau, 0]$. We
extend the discrete EM solution for the $G$-SDDE \eqref{eq:1} to the
continuous one by the following
\begin{eqnarray}\label{eq:3}
X(t)&=&\xi(0)+\int_0^t f(X_{\Delta}(r), X_{\Delta}(r-\tau))
dr+\int_0^t g(X_{\Delta}(r), X_{\Delta}(r-\tau)) d\langle
B\rangle(r)\nonumber
\\&&+\int_0^t h(X_{\Delta}(r),
X_{\Delta}(r-\tau))d B(r), \ t\geq 0.
\end{eqnarray}
Accordingly, the continuous EM approximate solution for the $G$-SDE
\eqref{eq:2} is defined by
\begin{eqnarray}\label{eq:4}
Y(t)&=&\xi(0)+\int_0^t f(Y_{\Delta}(r), Y_{\Delta}(r)) dr+\int_0^t
g(Y_{\Delta}(r), Y_{\Delta}(r)) d\langle B\rangle(r)\nonumber
\\&&+\int_0^t h(Y_{\Delta}(r), Y_{\Delta}(r))dB(r),\ t\geq 0,
\end{eqnarray}
where $Y_{\Delta}(t)=Y_n$, for $t \in\left[t_n, t_{n+1}\right)$. For
convenience, the equations \eqref{eq:3} and \eqref{eq:4} are
referred to as $G$-EMSDDE and $G$-EMSDE, respectively.

Let us now give the definitions of the $p$-th moment practical
exponential stability for the $G$-SDDE \eqref{eq:1}, the $G$-SDE
\eqref{eq:2} and their corresponding continuous EM methods.

\begin{definition} The $G$-SDDE \eqref{eq:1} and
$G$-SDE\eqref{eq:2} are said to be practically exponential stable in
$p$-th moment if there are positive numbers $M_1, \lambda_1, d_1$
and $M_2,\lambda_2, d_2$ such that
\begin{eqnarray}\label{def-pes-gsdde}
\mathbb{E}|x(t)|^p \leq M_1\|\xi\|_{\mathbb{E}}^p e^{-\lambda_1
t}+d_1,\ \forall t \geq 0
\end{eqnarray}
and
\begin{eqnarray}\label{def-pes-gsde}
\mathbb{E}|y(t)|^p \leq M_2 \mathbb{E}|\xi(0)|^p e^{-\lambda_2
t}+d_2,\ \forall t \geq 0
\end{eqnarray}
for all $\xi \in \mathcal{L}_{\mathcal{H}_0}^p\left([-\tau, 0] ;
\mathbb{R}^n\right)$.
\end{definition}
\begin{definition}
Given a step size $\Delta=\tau/m$ for a positive integer $m$, the
continuous EM methods $G$-EMSDDE \eqref{eq:3} and $G$-EMSDE
\eqref{eq:4} are said to be practically exponential stable in $p$-th
moment if there are positive numbers $L_1, \gamma_1, k_1$ and $L_2,
\gamma_2, k_2$ such that
\begin{eqnarray}\label{def-pes-EMgsdde}
\mathbb{E}|X(t)|^p \leq L_1\|\xi\|_{\mathbb{E}}^p e^{-\gamma_1
t}+k_1,\  \forall t \geq 0
\end{eqnarray}
and
\begin{eqnarray}\label{def-pes-EMgsde}
\mathbb{E}|Y(t)|^p \leq L_2 \mathbb{E}|\xi(0)|^p e^{-\gamma_2
t}+k_2,\ \forall t \geq 0
\end{eqnarray}
for all $\xi \in \mathcal{L}_{\mathcal{H}_0}^p\left([-\tau, 0] ;
\mathbb{R}^n\right)$.
\end{definition}
For more details of the following Burkholder-Davis-Gundy type
inequalities in the $G$-framework, see Gao \cite{Gao09} and Fei et.
al \cite{FFMY22}.
\begin{lemma}\label{lemma:2.1}
 For each $p \geq 1, \eta \in M_G^p([0, T])$ and $0 \leq
 t \leq T$, we have
$$
\mathbb{E}\left[\sup_{0\leq r \leq t}\left|\int_0^r \eta(s) d\langle
B\rangle(s)\right|^p\right] \leq \bar{\sigma}^{2p} t^{p-1} \int_0^t
\mathbb{E}\left(\left|\eta(s)\right|^p\right)ds.
$$
\end{lemma}

\begin{lemma}\label{lemma:2.2}
Let $ p>0, \eta\in M_G^p([0, T])$ and $0 \leq t \leq T$. Then,
$$
\mathbb{E}\left[\sup_{0 \leq r \leq t}\left|\int_0^r \eta(s)d
B(s)\right|^p\right] \leq C(p) \bar{\sigma}^p
\mathbb{E}\left[\left(\int_0^t \left|\eta(s)\right|^2
ds\right)^{\frac{p}{2}}\right].
$$
Especially, for $p\geq 2$,
$$
\mathbb{E}\left[\sup _{0 \leq r \leq t}\left|\int_0^r \eta(s) d
B_s\right|^p\right] \leq C(p)\bar{\sigma}^p   t^{\frac{p}{2}-1}
\int_0^t \mathbb{E}\left(\left|\eta(s)\right|^p\right)ds,
$$
where  the positive constant $C(p)$ is defined as follows
\begin{eqnarray*}
C(p)=\left\{ \begin{aligned}
        &\left(\frac{p^{p+1}}{2(p-1)^{p-1}}\right)^{\frac{p}{2}},\  p\geq 2, \\
                  &\left(\frac{32}{p}\right)^{\frac{p}{2}}, \ 0<p<2.
                          \end{aligned}\right.
                          \end{eqnarray*}
\end{lemma}

\section{Boundedness of $p$-th moment }

From now on, for given $s\geq0$, we will use $x(t; s, \eta)$ to
denote the solution to the $G$-SDDE \eqref{eq:1} with initial data
$x_s=\eta=\{\eta(s+r),-\tau \leq r \leq 0\} \in
\mathcal{L}_{\mathcal{H}_s}^p\left([-\tau, 0] ;
\mathbb{R}^n\right)$. Note that the initial data $s$ and $\eta$ will
be chosen appropriately in our proofs, in particular, if $s=0$, we
take $\eta=\xi$. Let's establish the $p$-th moment boundedness for
the $G$-SDDE \eqref{eq:1},  $G$-SDE \eqref{eq:2} and their
continuous EM methods, respectively.

\begin{lemma}\label{lemma:3.1}  Let $x$ (resp. $X$) be the solution for the $G$-SDDE \eqref{eq:1}
(resp.  $G$-EMSDDE \eqref{eq:2}) with initial data $\eta$, under
 (H1), for any $p\geq 2$, it holds that
$$
\begin{aligned}
&\sup_{s-\tau\leq t\leq T}\mathbb{E}|x(t; s, \eta)|^p\vee
\sup_{s-\tau\leq t\leq T}\mathbb{E}|X(t; s, \eta)|^p\\&\leq
\left[(\widehat{L}+\bar{\sigma}^2p\widehat{L})(T-s)+
(1+\widehat{L}\tau+\bar{\sigma}^2p\widehat{L}\tau)\|\eta\|_{\mathbb{E}}^{p}\right]
e^{\frac{p+3p\widehat{L}-2\widehat{L}+\bar{\sigma}^2(p+2p^2\widehat{L}-p\widehat{L})}{2}(T-s)},
\forall T>s\geq 0.
\end{aligned}
$$
\end{lemma}
\begin{proof} Write $x(t; s, \eta)=x(t)$, by the $G$-It\^{o}'s
formula, we have
$$
\begin{aligned}
d|x(r)|^p= & p|x(r)|^{p-2}\langle x(r), f(x(r), x(r-\tau))\rangle d
r +\left(\frac{p(p-2)}{2} |x(r)|^{p-4}|\langle x(r), h(x(r),
x(r-\tau))\rangle|^2\right.
 \\
& \left.+\frac{p}{2} |x(r)|^{p-2}|h(x(r), x(r-\tau))|^2 + p|x(r)|^{p-2}\langle x(r), g(x(r), x(r-\tau))\rangle \right) d\langle B\rangle(r) \\
& +p|x(r)|^{p-2}\langle x(r), h(x(r), x(r-\tau))\rangle dB(r).
\end{aligned}
$$
Using the elementary inequality and Lemma \ref{lemma:2.1}, we obtain
$$
\begin{aligned}
&\mathbb{E}|x(t)|^p-\mathbb{E}|x(s)|^p
 \\
 \leq &\frac{p}{2}\mathbb{E}\int_s^t|x(r)|^{p-2}(|x(r)|^2+|f(x(r), x(r-\tau))|^2) dr \\
& +\bar{\sigma}^2 \mathbb{E} \left[\int_s^t\left(\frac{p(p-2)}{2} |x(r)|^{p-4}|x(r)|^2|h(x(r), x(r-\tau))|^2\right.\right.\\
& \left.\left.+\frac{p}{2} |x(r)|^{p-2}|h(x(r),
x(r-\tau))|^2+\frac{p}{2}|x(r)|^{p-2}(|x(r)|^2+|g(x(r),
x(r-\tau))|^2)\right)dr\right].
\end{aligned}
$$
By \eqref{lgc:2} and the H\"{o}lder inequality, we derive that
$$
\begin{aligned}
\mathbb{E}|x(t)|^p-\mathbb{E}|x(s)|^p
\leq & \frac{p(1+\widehat{L})+\bar{\sigma}^2(p+p^2\widehat{L})}{2}\mathbb{E}\int_s^t|x(r)|^{p}dr\\
& +
\frac{p\widehat{L}+\bar{\sigma}^2p^2\widehat{L}}{2}\left(\mathbb{E}\int_s^t|x(r)|^{p-2}dr
+\mathbb{E}\int_s^t|x(r)|^{p-2}|x(r-\tau)|^2dr\right)
\\
  \leq &(\widehat{L}+\bar{\sigma}^2p\widehat{L})(t-s)
 +\frac{p+3p\widehat{L}-4\widehat{L}+\bar{\sigma}^2(p+2p^2\widehat{L}-3p\widehat{L})}{2}\mathbb{E}\int_s^t|x(r)|^{p}dr \\
& +
\tau(\widehat{L}+\bar{\sigma}^2p\widehat{L})\|\eta\|_{\mathbb{E}}^{p}
+(\widehat{L}+\bar{\sigma}^2p\widehat{L})\mathbb{E}\int_{s}^t|x(r)|^pdr\\
\leq &  (\widehat{L}+\bar{\sigma}^2p\widehat{L})(t-s)+
\tau(\widehat{L}+\bar{\sigma}^2p\widehat{L})\|\eta\|_{\mathbb{E}}^{p}
\\&
+\frac{p+3p\widehat{L}-2\widehat{L}+\bar{\sigma}^2(p+2p^2\widehat{L}-p\widehat{L})}{2}\mathbb{E}\int_s^t|x(r)|^{p}dr.
\end{aligned}
$$
Applying the Gronwall inequality yields
\begin{eqnarray}\label{ineq3.1}
\mathbb{E}|x(t)|^p\leq
\left[(\widehat{L}+\bar{\sigma}^2p\widehat{L})(t-s)+
(1+\widehat{L}\tau+\bar{\sigma}^2p\widehat{L}\tau)\|\eta\|_{\mathbb{E}}^{p}\right]
e^{\frac{p+3p\widehat{L}-2\widehat{L}+\bar{\sigma}^2(p+2p^2\widehat{L}-p\widehat{L})}{2}(t-s)}.
\end{eqnarray}
Thus,
$$
\begin{aligned}
\sup_{s-\tau\leq t\leq T}\mathbb{E}|x(t)|^p\leq&
\left[(\widehat{L}+\bar{\sigma}^2p\widehat{L})(T-s)+
(1+\widehat{L}\tau+\bar{\sigma}^2p\widehat{L}\tau)\|\eta\|_{\mathbb{E}}^{p}\right]
e^{\frac{p+3p\widehat{L}-2\widehat{L}+\bar{\sigma}^2(p+2p^2\widehat{L}-p\widehat{L})}{2}(T-s)}.
\end{aligned}
$$
Similarly, write $X(t; s, \eta)=X(t)$, we can show that
$$
\begin{aligned}
\sup_{s-\tau\leq t\leq T}\mathbb{E}|X(t)|^p\leq&
\left[(\widehat{L}+\bar{\sigma}^2p\widehat{L})(T-s)+
(1+\widehat{L}\tau+\bar{\sigma}^2p\widehat{L}\tau)\|\eta\|_{\mathbb{E}}^{p}\right]
e^{\frac{p+3p\widehat{L}-2\widehat{L}+\bar{\sigma}^2(p+2p^2\widehat{L}-p\widehat{L})}{2}(T-s)}.
\end{aligned}
$$
 The proof is complete.
\end{proof}
\begin{lemma}\label{lemma:3.2}
Let $y$ (resp. $Y$) be the solution for the $G$-SDE \eqref{eq:3}
(resp.  $G$-EMSDE \eqref{eq:4}) with initial data $\eta(s)$, under
(H1), for any $p\geq 2$, it holds that
$$
\begin{aligned}
&\sup_{s\leq t\leq T}\mathbb{E}|y(t; s, \eta(s))|^p \vee \sup_{s\leq
t\leq T}\mathbb{E}|Y(t; s, \eta(s))|^p \\&
\leq\left[\mathbb{E}|\eta(s)|^p+(\widehat{L}+\bar{\sigma}^2p\widehat{L})(T-s)\right]
e^{\frac{p+3p\widehat{L}-2\widehat{L}+\bar{\sigma}^2p(1-2\widehat{L}+3p\widehat{L})}{2}(T-s)},
\forall T>s\geq 0.
\end{aligned}
$$
\end{lemma}
\begin{proof}  Write $y(t; s, \eta(s))=y(t)$, by the same procedure
as in the proof of Lemma \ref{lemma:3.1}, we have
$$
\begin{aligned}
\mathbb{E}|y(t)|^p-\mathbb{E}|y(s)|^p \leq
(\widehat{L}+\bar{\sigma}^2p\widehat{L})(t-s)+
\frac{p+3p\widehat{L}-2\widehat{L}+\bar{\sigma}^2p(1-2\widehat{L}+3p\widehat{L})}{2}\mathbb{E}\int_s^t|y(r)|^{p}dr.
\end{aligned}
$$
Applying the Gronwall inequality, we get
$$
\begin{aligned}
\mathbb{E}|y(t)|^p\leq
\left[\mathbb{E}|\eta(s)|^p+(\widehat{L}+\bar{\sigma}^2p\widehat{L})(t-s)\right]
e^{\frac{p+3p\widehat{L}-2\widehat{L}+\bar{\sigma}^2p(1-2\widehat{L}+3p\widehat{L})}{2}(t-s)}.
\end{aligned}
$$
As a result,
$$
\begin{aligned}
\sup_{s\leq t\leq T}\mathbb{E}|y(t)|^p\leq
\left[\mathbb{E}|\eta(s)|^p+(\widehat{L}+\bar{\sigma}^2p\widehat{L})(T-s)\right]
e^{\frac{p+3p\widehat{L}-2\widehat{L}+\bar{\sigma}^2p(1-2\widehat{L}+3p\widehat{L})}{2}(T-s)}.
\end{aligned}
$$
Similarly, write $Y(t; s, \eta(s))=Y(t)$, repeating the procedures
above we can show that
$$
\begin{aligned}
\sup_{s\leq t\leq T}\mathbb{E}|Y(t)|^p\leq
\left[\mathbb{E}|\eta(s)|^p+(\widehat{L}+\bar{\sigma}^2p\widehat{L})(T-s)\right]
e^{\frac{p+3p\widehat{L}-2\widehat{L}+\bar{\sigma}^2p(1-2\widehat{L}+3p\widehat{L})}{2}(T-s)}.
\end{aligned}
$$
The proof is complete.
\end{proof}

\section{Main results}
In this section, we will prove that the equivalence of practical
exponential stability in $p$-th moment sense among the $G$-SDDE, the
auxiliary $G$-SDE and their corresponding EM methods. The whole
section will be divided into four subsections.

\subsection{$G$-SDE shares stability with $G$-SDDE}

\begin{lemma}\label{lemma:3.3} Let $x$ (resp. $X$) be the solution of the $G$-SDDE \eqref{eq:1}
(resp.  $G$-EMSDDE \eqref{eq:3}). Under (H1), for any $T>s\geq 0$,
it holds that
$$
\begin{aligned}
& \sup_{s+\tau \leq t \leq T} \mathbb{E}\left|x(t ; s,
\eta)-x(t-\tau ; s, \eta)\right|^p \ \vee \sup_{s+\tau \leq t \leq
T} \mathbb{E}\left|X(t ; s, \eta)-X(t-\tau ; s, \eta)\right|^p  \\
&\leq  \tau^{\frac{p}{2}}\left(K_1 \|\eta\|_{\mathbb{E}}^{p}
e^{\frac{p+2p\widehat{L}+\bar{\sigma}^2(p+p^2\widehat{L}+p\widehat{L})}{2}(T-s)}+N_1\right),\\
&\sup_{s \leq t \leq s+\tau} \mathbb{E}\left|x(t ; s, \eta)-x(t-\tau
; s, \eta)\right|^p \  \vee \sup_{s \leq t \leq s+\tau}
\mathbb{E}\left|X(t ; s, \eta)-X(t-\tau ; s, \eta)\right|^p \leq
 K_2\|\eta\|_{\mathbb{E}}^{p}+N_2,
\end{aligned}
$$
where
$$
\begin{aligned}
& K_1=3^{\frac{3p}{2}-1}
\widehat{L}^{\frac{p}{2}}(1+\widehat{L}\tau+\bar{\sigma}^2p\widehat{L}\tau)\left[\tau^{\frac{p}{2}}+\bar{\sigma}^{2p}\tau^{\frac{p}{2}}
+\left(\frac{p^{p+1}}{2(p-1)^{p-1}}\right)^{\frac{p}{2}}\bar{\sigma}^{p}\right], \\
& N_1=N_1(s, T)=3^{\frac{3p}{2}-2} \widehat{L}^{\frac{p}{2}}
\left[1+2(\widehat{L}+\bar{\sigma}^2p\widehat{L})(T-s)
e^{\frac{p+2p\widehat{L}+\bar{\sigma}^2(p+p^2\widehat{L}+p\widehat{L})}{2}(T-s)}\right]
\\
&\ \ \qquad \times
\left[\tau^{\frac{p}{2}}+\bar{\sigma}^{2p}\tau^{\frac{p}{2}}
+\left(\frac{p^{p+1}}{2(p-1)^{p-1}}\right)^{\frac{p}{2}}\bar{\sigma}^{p}\right],\\
&
K_2=2^{p-1}\left[1+(1+\widehat{L}\tau+\bar{\sigma}^2p\widehat{L}\tau)
e^{\frac{p+3p\widehat{L}-2\widehat{L}+\bar{\sigma}^2(p+2p^2\widehat{L}-p\widehat{L})}{2}\tau}\right], \\
& N_2=2^{p-1} \tau(\widehat{L}+\bar{\sigma}^2p\widehat{L})
e^{\frac{p+3p\widehat{L}-2\widehat{L}+\bar{\sigma}^2(p+2p^2\widehat{L}-p\widehat{L})}{2}\tau}.
\end{aligned}
$$
\end{lemma}
\begin{remark}
We emphasis that the notion $N_1=N_1(s, T)$ depends on the interval
length $T-s$ and therefore, in the following text, the variables $s$
and $T$ in $N_1$ should be replaced appropriately according to the
corresponding time interval.
\end{remark}

\begin{proof} Write $x(t; s, \eta)=x(t)$. For $t \geq
s+\tau$, applying the elementary inequality, the H\"{o}lder
inequality, Lemmas \ref{lemma:2.1} and \ref{lemma:2.2}, we get that
$$
\begin{aligned}
\mathbb{E}|x(t)-x(t-\tau)|^p
 \leq &  3^{p-1}\left[\tau^{p-1} \mathbb{E} \int_{t-\tau}^t|f(x(r),
x(r-\tau))|^p dr+\bar{\sigma}^{2p}\tau^{p-1} \mathbb{E}
\int_{t-\tau}^t|g(x(r), x(r-\tau))|^p dr\right.
\\&
\left.+\left(\frac{p^{p+1}}{2(p-1)^{p-1}}\right)^{\frac{p}{2}}\bar{\sigma}^{p}\tau^{\frac{p}{2}-1}
\mathbb{E}\int_{t-\tau}^t |h(x(r), x(r-\tau))|^pdr\right].
\end{aligned}
$$
By \eqref{lgc:2}, we have
\begin{eqnarray}\label{pthlg}
|f(x(t), x(t-\tau))|^p \vee|g(x(t), x(t-\tau))|^p \vee|h(x(t),
x(t-\tau))|^p\leq 3^{\frac{p}{2}-1}
\widehat{L}^{\frac{p}{2}}\left(1+|x(t)|^p+|x(t-\tau)|^p\right).
\end{eqnarray}
Thus,
$$
\begin{aligned}
& \mathbb{E}|x(t)-x(t-\tau)|^p
\\
& \leq 3^{\frac{3p}{2}-2}
\widehat{L}^{\frac{p}{2}}\left[\tau^{p}+\bar{\sigma}^{2p}\tau^{p}
+\left(\frac{p^{p+1}}{2(p-1)^{p-1}}\right)^{\frac{p}{2}}\bar{\sigma}^{p}\tau^{\frac{p}{2}}
\right]\left(1+2\sup_{s-\tau\leq u\leq t}|x(u)|^p\right).
\end{aligned}
$$
By Lemma \ref{lemma:3.1} , we obtain that
$$
\begin{aligned}
& \sup_{s+\tau\leq t\leq T}\mathbb{E}|x(t)-x(t-\tau)|^p
\\
& \leq 3^{\frac{3p}{2}-2}
\widehat{L}^{\frac{p}{2}}\left[\tau^{p}+\bar{\sigma}^{2p}\tau^{p}
+\left(\frac{p^{p+1}}{2(p-1)^{p-1}}\right)^{\frac{p}{2}}\bar{\sigma}^{p}\tau^{\frac{p}{2}}
\right]\\&
\quad\times\left[1+2\left((\widehat{L}+\bar{\sigma}^2p\widehat{L})(t-s)
+(1+\widehat{L}\tau+\bar{\sigma}^2p\widehat{L}\tau)\|\eta\|_{\mathbb{E}}^{p}\right)
e^{\frac{p+2p\widehat{L}+\bar{\sigma}^2(p+p^2\widehat{L}+p\widehat{L})}{2}(T-s)}\right]
\\
& \leq \tau^{\frac{p}{2}} \left(K_1  \|\eta\|_{\mathbb{E}}^{p}
e^{\frac{p+2p\widehat{L}+\bar{\sigma}^2(p+p^2\widehat{L}+p\widehat{L})}{2}(T-s)}+N_1\right).
\end{aligned}
$$
When $t \in[s, s+\tau]$, by \eqref{ineq3.1}, we have
$$
\begin{aligned}
& \mathbb{E}|x(t)-x(t-\tau)|^p \leq 2^{p-1} \mathbb{E}|x(t)|^p+2^{p-1} \mathbb{E}|x(t-\tau)|^p \\
&\leq 2^{p-1} \left[(\widehat{L}+\bar{\sigma}^2p\widehat{L})(t-s)+
(1+\widehat{L}\tau+\bar{\sigma}^2p\widehat{L}\tau)\|\eta\|_{\mathbb{E}}^{p}\right]
e^{\frac{p+3p\widehat{L}-2\widehat{L}+\bar{\sigma}^2(p+2p^2\widehat{L}-p\widehat{L})}{2}(t-s)}
+2^{p-1}\|\eta\|_{\mathbb{E}}^{p}
\end{aligned}
$$
and then
$$
\begin{aligned}
\sup_{s\leq t\leq s+\tau}\mathbb{E}|x(t)-x(t-\tau)|^p \leq
K_2\|\eta\|_{\mathbb{E}}^{p}+N_2.
\end{aligned}
$$
By the same procedure above, we can get the same result for $X(t; s,
\eta)=X(t)$. The proof is complete.
\end{proof}

\begin{lemma}\label{lamma:3.4}
Let $x, y$ be the solution to the $G$-SDDE \eqref{eq:1} and its
auxiliary  $G$-SDE \eqref{eq:2}, respectively. Under (H1), for any
$T> s \geq 0$, it holds that
$$
\begin{aligned}
&\sup_{s\leq t\leq T}\mathbb{E}|x(t)-y(t)|^p
\\
\leq &(2L+2\bar{\sigma}^2pL)\left[\left(K_2
\|\eta\|_{\mathbb{E}}^{p}+N_2 \right)\tau+ \tau^{\frac{p}{2}}
\left(K_1\|\eta\|_{\mathbb{E}}^{p}
e^{\frac{p+2p\widehat{L}+\bar{\sigma}^2(p+p^2\widehat{L}+p\widehat{L})}{2}(T-s)}+N_1(T-s)\right)\right]
\\&
\times e^{\frac{p+5pL-4L+\bar{\sigma}^2(p+5p^2L-4pL)}{2}(T-s)}.
\end{aligned}
$$
\end{lemma}

\begin{proof}
Write $x(t ; s, \eta)=x(t)$ and $y(t; s, \eta(s))=y(t)$, then
$$
\begin{aligned}
 x(t)-y(t) = &  \int_{s}^t(f(x(r),
x(r-\tau))-f(y(r), y(r)))dr+\int_{s}^t(g(x(r), x(r-\tau))-g(y(r),
y(r)))d\langle B\rangle(r)  \\
&+\int_{s}^t(h(x(r), x(r-\tau))-h(y(r), y(r)))dB(r) .
\end{aligned}
$$
By the $G$-It\^{o}'s formula, the elementary inequality and Lemma
\ref{lemma:2.1}, we have
$$
\begin{aligned}
&\mathbb{E}|x(t)-y(t)|^p \\
\leq &\frac{p}{2}\mathbb{E}\int_s^t|x(r)-y(r)|^{p-2}(|x(r)-y(r)|^2+|f(x(r), x(r-\tau))-f(y(r), y(r))|^2) dr \\
& +\bar{\sigma}^2 \mathbb{E}\left [\int_s^t\left(\frac{p(p-2)}{2} |x(r)-y(r)|^{p-4}|x(r)-y(r)|^2|h(x(r), x(r-\tau))-h(y(r), y(r))|^2\right.\right. \\
& +\frac{p}{2} |x(r)-y(r)|^{p-2}|h(x(r), x(r-\tau))-h(y(r), y(r))|^2 \\
&  \left.\left.+\frac{p}{2}|x(r)-y(r)|^{p-2}(|x(r)-y(r)|^2+|g(x(r),
x(r-\tau))-g(y(r), y(r))|^2) \right) dr \right].
\end{aligned}
$$
Combining  (H1) with the H\"{o}lder inequality, we can get
$$
\begin{aligned}
&\mathbb{E}|x(t)-y(t)|^p
 \\
\leq &
\frac{p+pL}{2}\mathbb{E}\int_s^t|x(r)-y(r)|^{p}dr+\frac{pL}{2}\mathbb{E}\int_s^t|x(r)-y(r)|^{p-2}|x(r-\tau)-y(r)|^2
dr
\\
& +\bar{\sigma}^2 \mathbb{E}\left(\int_s^t \frac{p+p^2L}{2}
|x(r)-y(r)|^{p}dr
+\frac{p^2L}{2}\int_s^t|x(r)-y(r)|^{p-2}|x(r-\tau)-y(r)|^2dr\right)
\\
\leq & \frac{p+3pL}{2}\mathbb{E}\int_s^t|x(r)-y(r)|^{p}dr+pL\mathbb{E}\int_s^t|x(r)-y(r)|^{p-2}|x(r)-x(r-\tau)|^2 dr \\
& +\bar{\sigma}^2 \mathbb{E}\left(\int_s^t \frac{p+3p^2L}{2}
|x(r)-y(r)|^{p}dr +p^2L\int_s^t|x(r)-y(r)|^{p-2}|x(r)-x(r-\tau)|^2
dr\right)
\\
\leq
&\frac{p+5pL-4L+\bar{\sigma}^2(p+5p^2L-4pL)}{2}\mathbb{E}\int_s^t|x(r)-y(r)|^{p}dr
\\ &+(2L+2\bar{\sigma}^2pL) \mathbb{E}\int_s^t |x(r)-x(r-\tau)|^p
dr
\\
\leq
&\frac{p+5pL-4L+\bar{\sigma}^2(p+5p^2L-4pL)}{2}\mathbb{E}\int_s^t|x(r)-y(r)|^{p}dr
\\ &+(2L+2\bar{\sigma}^2pL)\left[\left(K_2\|\eta\|_{\mathbb{E}}^{p}+N_2\right)\tau+ \tau^{\frac{p}{2}} \left(K_1\|\eta\|_{\mathbb{E}}^{p}
e^{\frac{p+2p\widehat{L}+\bar{\sigma}^2(p+p^2\widehat{L}+p\widehat{L})}{2}(t-s)}+N_1(s,
t)(t-s)\right)\right].
\end{aligned}
$$
The desired assertion then follows by applying the Gronwall
inequality.
\end{proof}
The following theorem shows that the auxiliary $G$-SDE \eqref{eq:2}
shares the practical exponential stability with the $G$-SDDE
\eqref{eq:1}, provided the delay $\tau$ is small enough.
\begin{theorem} \label{th:4} Assume that the $G$-SDDE \eqref{eq:1} is  practically exponential
stable in $p$-th moment and that (H1) holds. For any given $\delta
\in(0,1)$, let $T=\frac{\ln\left(2^{p-1} M_1 /
\delta\right)}{\lambda_1}+\tau$, if $\tau>0$ is small enough such
that
$$
R(\tau):=\delta+2^{p}
(L+\bar{\sigma}^2pL)\left(K_2\tau+K_1\tau^{\frac{p}{2}}
e^{\frac{p+2p\widehat{L}+\bar{\sigma}^2(p+p^2\widehat{L}+p\widehat{L})}{2}(2T-\tau)}
\right)e^{\frac{p+5pL-4L+\bar{\sigma}^2(p+5p^2L-4pL)}{2}(2T-\tau)}<1,
$$
then  the $G$-SDE \eqref{eq:2}  is also practically exponential
stable in $p$-th moment.
\end{theorem}
\begin{proof}
Write $y(t)=y(t ; 0, \xi(0)), x(t)=x(t ; 0, \xi)$, by the elementary
inequality,  the practical exponential stability in $p$-th moment of
the $G$-SDDE \eqref{eq:1} and Lemma \ref{lamma:3.4}, we have
\begin{eqnarray}\label{ineq3.3}
\mathbb{E}|y(t)|^p &\leq &  2^{p-1} (\mathbb{E}|x(t)|^p+ \mathbb{E}|x(t)-y(t)|^p) \nonumber \\
&\leq & 2^{p-1}  M_1\|\xi\|_{\mathbb{E}}^{p} e^{-\lambda_1 t}+
 2^{p-1}d_1
\nonumber
\\
&& +2^{p}
(L+\bar{\sigma}^2pL)\left(\left(K_2\|\xi\|_{\mathbb{E}}^{p}+N_2\right)\tau+
\tau^{\frac{p}{2}} \left(K_1\|\xi\|_{\mathbb{E}}^{p}
e^{\frac{p+2p\widehat{L}+\bar{\sigma}^2(p+p^2\widehat{L}+p\widehat{L})}{2}t}+N_1(0,
t)t\right)\right)\nonumber
\\
&&\times e^{\frac{p+5pL-4L+\bar{\sigma}^2(p+5p^2L-4pL)}{2}t}.
\end{eqnarray}
By the definition of $T$, we observe that
$$
2^{p-1} M_1 e^{-\lambda_1(T-\tau)}=\delta .
$$
Thus
\begin{eqnarray*}
&&\sup _{T-\tau \leq t \leq 2T-\tau}\mathbb{E}|y(t)|^p
\nonumber\\&\leq & \left[\delta+2^{p}
(L+\bar{\sigma}^2pL)\left(K_2\tau+K_1\tau^{\frac{p}{2}}
e^{\frac{p+2p\widehat{L}+\bar{\sigma}^2(p+p^2\widehat{L}+p\widehat{L})}{2}(2T-\tau)}
\right)e^{\frac{p+5pL-4L+\bar{\sigma}^2(p+5p^2L-4pL)}{2}(2T-\tau)}\right]\|\xi\|_{\mathbb{E}}^{p}\nonumber
\\
&& +
 2^{p-1}d_1+2^{p}
(L+\bar{\sigma}^2pL)(N_2\tau+N_1(0, T)\tau^{\frac{p}{2}}(2T-\tau))
e^{\frac{p+5pL-4L+\bar{\sigma}^2(p+5p^2L-4pL)}{2}(2T-\tau)}\nonumber
\\
&=&R(\tau)\|\xi\|_{\mathbb{E}}^{p}+d_3,
\end{eqnarray*}
where $d_3=2^{p-1}d_1+2^{p} (L+\bar{\sigma}^2pL)(N_2\tau+N_1(0,
T)\tau^{\frac{p}{2}}(2T-\tau))
e^{\frac{p+5pL-4L+\bar{\sigma}^2(p+5p^2L-4pL)}{2}(2T-\tau)}$. Since
$R(\tau)<1$,  choosing $\lambda_2>0$ such that
$R(\tau)=e^{-\lambda_2 T}$, we have
$$
\sup _{T-\tau \leq t \leq 2T-\tau}\mathbb{E}|y(t)|^p\leq
e^{-\lambda_2 T}\|\xi\|_{\mathbb{E}}^{p} +d_3.
$$
Using the flow property of the solution to the $G$-SDE \eqref{eq:2},
for $y(t ; 0, \xi(0))=y(t ; iT, y(iT)), i=0,1,2, \cdots$, we write
$x(t ; iT, y_{iT})=x(t)$, then the $p$-th moment practical
exponential stability of the $G$-SDDE \eqref{eq:1} implies that
$$
\mathbb{E}|x(t)|^p \leq  M_1\|y_{iT}\|_{\mathbb{E}}^{p}
 e^{-\lambda_1(t-iT)}+d_1,
$$
where $\|y_{iT}\|_{\mathbb{E}}^{p}=\sup _{iT-\tau \leq u \leq iT}
\mathbb{E}|y(u)|^p$. For any $t \geq i T$, using Lemma
\ref{lamma:3.4} again, we have
\begin{eqnarray*}
\mathbb{E}|y(t)|^p&\leq &  2^{p-1} \mathbb{E}|x(t; iT,
y_{iT})|^p+2^{p-1} \mathbb{E}|x(t ; iT, y_{iT})-y(t ; iT, y(iT))|^p \nonumber \\
&\leq & 2^{p-1} M_1 \left\|y_{i T}\right\|_{\mathbb{E}}^{p} e^{-\lambda_1(t-iT)}
+ 2^{p-1}d_1 +2^{p}(L+\bar{\sigma}^2pL)\nonumber \\
 &&\times\left[\left(K_2\left\|y_{i T}\right\|_{\mathbb{E}}^{p}+N_2\right)\tau+
\tau^{\frac{p}{2}} \left(K_1\left\|y_{i T}\right\|_{\mathbb{E}}^{p}
e^{\frac{p+2p\widehat{L}+\bar{\sigma}^2(p+p^2\widehat{L}+p\widehat{L})}{2}(t-iT)}+N_1(iT,
t)(t-iT)\right)\right] \nonumber \\&& \times
e^{\frac{p+5pL-4L+\bar{\sigma}^2(p+5p^2L-4pL)}{2}(t-iT)}.
\end{eqnarray*}
Hence
\begin{eqnarray*}
&&\sup _{(i+1)T-\tau \leq t \leq(i+2)T-\tau}\mathbb{E}|y(t)|^p
\nonumber \\&\leq &
 2^{p-1} M_1 \|y_{iT}\|_{\mathbb{E}}^{p} e^{-\lambda_1(T-\tau)} + 2^{p-1}d_1+2^{p}(L+\bar{\sigma}^2pL)\nonumber \\
 &&\times\left[\left(K_2\|y_{iT}\|_{\mathbb{E}}^{p}+N_2\right)\tau+
\tau^{\frac{p}{2}} \left(K_1\left\|y_{i T}\right\|_{\mathbb{E}}^{p}
e^{\frac{p+2p\widehat{L}+\bar{\sigma}^2(p+p^2\widehat{L}+p\widehat{L})}{2}(2T-\tau)}+N_1(0,
T)(2T-\tau)\right)\right] \nonumber \\&& \times
e^{\frac{p+5pL-4L+\bar{\sigma}^2(p+5p^2L-4pL)}{2}(2T-\tau)}\nonumber\\
&\leq & R(\tau)\|y_{iT}\|_{\mathbb{E}}^{p} + d_3
\nonumber\\
&\leq & e^{-\lambda_2 T} \sup _{iT-\tau \leq t \leq
iT}\mathbb{E}|y(t)|^P  + d_3.
\end{eqnarray*}
Recalling that $T \geq \tau$, by iteration, we obtain that
\begin{eqnarray*}
&&\sup_{(i+1) T-\tau \leq t \leq(i+2) T-\tau}
\mathbb{E}|y(t)|^P\nonumber
\\
&& \leq e^{-i\lambda_2 T}  \sup _{T-\tau \leq t \leq
2T-\tau}\mathbb{E}|y(t)|^P+d_3(e^{-\lambda_2(i-1)T}+e^{-\lambda_2(i-2)T}+\cdots+e^{-\lambda_2T}+1)
 \nonumber\\
&& \leq e^{-(i+1)\lambda_2 T}
\|\xi\|_{\mathbb{E}}^{p}+d_3(e^{-\lambda_2iT}+e^{-\lambda_2(i-1)T}+e^{-\lambda_2(i-2)T}+\cdots+e^{-\lambda_2T}+1)
 \nonumber\\
&& \leq e^{-(i+1)\lambda_2 T}
\|\xi\|_{\mathbb{E}}^{p}+\frac{d_3}{1-e^{-\lambda_2T}}.
\end{eqnarray*}
Moreover, using \eqref{ineq3.3}, we have
$$
\begin{aligned}
\sup _{0 \leq t \leq T-\tau} \mathbb{E}|y(t)|^P \leq & 2^{p-1}
M_1\|\xi\|_{\mathbb{E}}^{p} +
 2^{p-1}d_1
 +2^{p}(L+\bar{\sigma}^2pL)
\\
&\times\left[\tau\left(K_2\|\xi\|_{\mathbb{E}}^{p}+N_2\right)+
\tau^{\frac{p}{2}} \left(K_1\|\xi\|_{\mathbb{E}}^{p}
e^{\frac{p+2p\widehat{L}+\bar{\sigma}^2(p+p^2\widehat{L}+p\widehat{L})}{2}(T-\tau)}+N_1(T-\tau)\right)\right]\nonumber
\\
&\times e^{\frac{p+5pL-4L+\bar{\sigma}^2(p+5p^2L-4pL)}{2}(T-\tau)}
\\
\leq & 2^{p-1} M_1 \|\xi\|_{\mathbb{E}}^{p}+R(\tau)\|\xi\|_E^p + d_3
\\
\leq & (2^{p-1} M_1e^{\lambda_2T}+1)
\|\xi\|_{\mathbb{E}}^{p}e^{-\lambda_2T} + d_3.
\end{aligned}
$$
This, together with \eqref{ineq3.3}, implies that
$$
\mathbb{E}|y(t)|^p \leq M_2 \mathbb{E}|\xi(0)|^{p}e^{-\lambda_2 t} +
d_2, \ \forall t \geq 0,
$$
where $M_2= \frac{(2^{p-1}
M_1+1)\|\xi\|_{\mathbb{E}}^{p}e^{\lambda_2T}}{\mathbb{E}|\xi(0)|^{p}}$
and $d_2=\frac{d_3}{1-e^{-\lambda_2T}}$, i.e., the $G$-SDE
\eqref{eq:2} is practically exponential stable in $p$-th moment. We
complete the proof.
\end{proof}

\subsection{$G$-EMSDE shares stability with $G$-SDE}

\begin{lemma}\label{lemma:3.5}
Let $Y$ be the solution of the $G$-EMSDE \eqref{eq:4} and let $s
\geq 0$ be a multiple of $\Delta$. Under  (H1), it holds that
$$
\begin{aligned}
\sup_{s \leq t \leq T}\mathbb{E}\left|Y(t)-Y_{\Delta}(t)\right|^p
\leq &
D_1\left[1+2\left(\mathbb{E}|\eta(s)|^p+(\widehat{L}+\bar{\sigma}^2p\widehat{L})(T-s)\right)
e^{\frac{p+3p\widehat{L}-2\widehat{L}+\bar{\sigma}^2p(1-2\widehat{L}+3p\widehat{L})}{2}(T-s)}
\right]\Delta^{\frac{p}{2}},
\end{aligned}
$$
where $D_1=3^{\frac{3 p}{2}-2}\widehat{L}^{\frac{p}{2}}
\left(\tau^{\frac{p}{2}}+\bar{\sigma}^{2p}\tau^{\frac{p}{2}}+\bar{\sigma}^{p}\sqrt{(2p-1)!!
}\right)$.
\end{lemma}
\begin{proof}
We write $Y(t ; s, \eta(s))=Y(t)$ and $Y_{\Delta}(t ; s,
\eta(s))=Y_{\Delta}(t)$. Since $s$ is  a multiple of $\Delta$, we
assume that $s=n_s \Delta$, $n_s \in \mathbb{N}$. Given $t \geq s$,
let $n$ be the integer for which $t \in [s+n \Delta, s+(n+1)
\Delta)$. Then
$$
\begin{aligned}
Y(t)-Y_{\Delta}(t) & =Y(t)-Y_{n+n_s} \\
& =\int_{s+n \Delta}^t f\left(Y_{n+n_s}, Y_{n+n_s}\right) \mathrm{d}
r+\int_{s+n \Delta}^t g\left(Y_{n+n_s}, Y_{n+n_s}\right) \mathrm{d}
\langle B\rangle(r) +\int_{s+n \Delta}^t h\left(Y_{n+n_s},
Y_{n+n_s}\right) \mathrm{d} B(r) .
\end{aligned}
$$
With the help of the fundamental inequality $|a+b+c|^p \leq
3^{p-1}\left(|a|^p+|b|^p+|c|^p\right)$ for $p \geq 1$,  the
H\"{o}lder inequality, \eqref{pthlg} and Theorem 5.3 in Peng
\cite{Peng10}, we derive that
$$
\begin{aligned}
\mathbb{E}\left|Y(t)-Y_{\Delta}(t)\right|^p \leq & 3^{p-1} \mathbb{E}\left|f\left(Y_{n+n_s}, Y_{n+n_s}\right)\right|^p(t-(s+n \Delta))^p \\
& +3^{p-1} \mathbb{E}\left|g\left(Y_{n+n_s}, Y_{n+n_s}\right)\right|^p\left(\mathbb{E}|\langle B\rangle(t)-\langle B\rangle(s+n \Delta)|^{2 p}\right)^{\frac{1}{2}} \\
& +3^{p-1} \mathbb{E}\left|h\left(Y_{n+n_s}, Y_{n+n_s}\right)\right|^p\left(\mathbb{E}|B(t)-B(s+n \Delta)|^{2 p}\right)^{\frac{1}{2}}\\
\leq&  3^{p-1} \mathbb{E}\left|f\left(Y_{n+n_s},
Y_{n+n_s}\right)\right|^p \Delta^p
+3^{p-1} \sqrt{(2p-1)!! \bar{\sigma}^{2p}\Delta^{p}} \mathbb{E}\left|g\left(Y_{n+n_s}, Y_{n+n_s}\right)\right|^p  \\
&+3^{p-1} \sqrt{\bar{\sigma}^{4p}(t-(s+n \Delta))^{2p}}
\mathbb{E}\left|h\left(Y_{n+n_s}, Y_{n+n_s}\right)\right|^p
 \\
\leq & 3^{\frac{3 p}{2}-2}\widehat{L}^{\frac{p}{2}}
\left(\Delta^{p}+\bar{\sigma}^{2p}\Delta^{p}+\bar{\sigma}^{p}\sqrt{(2p-1)!!
}\Delta^{\frac{p}{2}}\right)
\left(1+2\mathbb{E}\left|Y_{n+n_s}\right|^p\right).
 \end{aligned}
$$
By Lemma \ref{lemma:3.2}, we obtain that
$$
\begin{aligned}
\sup _{s \leq t \leq T}\mathbb{E}\left|Y(t)-Y_{\Delta}(t)\right|^p
\leq &
D_1\left[1+2\left(\mathbb{E}|\eta(s)|^p+(\widehat{L}+\bar{\sigma}^2p\widehat{L})(T-s)\right)
e^{\frac{p+3p\widehat{L}-2\widehat{L}+\bar{\sigma}^2p(1-2\widehat{L}+3p\widehat{L})}{2}(T-s)}
\right]\Delta^{\frac{p}{2}}.
\end{aligned}
$$
The proof is complete.
\end{proof}

\begin{lemma}\label{lemma:3.6}  Let $y$ and $Y$ be the solution of the $G$-SDE \eqref{eq:2} and the
$G$-EMSDE \eqref{eq:4}, respectively. Assume that $s \geq 0$ is a
multiple of $\Delta$. Under  (H1), for any $T>s \geq 0$, it holds
that
$$\begin{aligned}
 \sup_{s \leq t \leq T}\mathbb{E}|y(t)-Y(t)|^p
 \leq &
4D_1(L+\bar{\sigma}^2pL)\Delta^{\frac{p}{2}}
e^{\left(4pL-4L+\frac{p}{2}+\bar{\sigma}^2(\frac{p}{2}+4p^2L-4pL)\right)(T-s)}
\\
&\times
\left[(T-s)+2\left(\mathbb{E}|\eta(s)|^p+(\widehat{L}+\bar{\sigma}^2p\widehat{L})(T-s)\right)
e^{\frac{p+3p\widehat{L}-2\widehat{L}+\bar{\sigma}^2p(1-2\widehat{L}+3p\widehat{L})}{2}(T-s)}
\right].
\end{aligned}$$
 \end{lemma}
\begin{proof} Write $y(t ; s, \eta(s))=y(t), Y(t ; s, \eta(s))=Y(t)$. For any
$t \in[s, T]$, we have
$$
\begin{aligned}
y(t)-Y(t)= & \int_s^t (f(y(r), y(r))-f (Y_{\Delta}(r),
Y_{\Delta}(r))) dr +\int_s^t (g(y(r), y(r))-g (Y_{\Delta}(r),
Y_{\Delta}(r)) d
\langle B \rangle(r)\\
&
 +\int_s^t(h(y(r),
y(r))-h(Y_{\Delta}(r), Y_{\Delta}(r))d B(r).
\end{aligned}
$$
By the $G$-It\^{o} formula, the elementary inequality, (H1) and
Lemmas \ref{lemma:2.1} and \ref{lemma:2.2}, we get
$$
\begin{aligned}
\mathbb{E}|y(t)-Y(t)|^p
\leq &  \frac{p}{2}\mathbb{E}\int_s^t|y(r)-Y(r)|^{p-2}\left(|y(r)-Y(r)|^2+ 2L|y(r)-Y_{\Delta}(r)|^2\right) dr \\
& +\bar{\sigma}^2 \mathbb{E}\left[\int_s^t \left(p(p-2)L|y(r)-Y(r)|^{p-2} |y(r)-Y_{\Delta}(r)|^2 \right.\right. \\
& +pL|y(r)-Y(r)|^{p-2} |y(r)-Y_{\Delta}(r)|^2  \\
& \left.\left.+\frac{p}{2}|y(r)-Y(r)|^{p-2}(|y(r)-Y(r)|^2+
2L|y(r)-Y_{\Delta}(r)|^2 )\right) dr\right]
\\
\leq &
\left[\frac{p}{2}(1+\bar{\sigma}^2)+2(pL+\bar{\sigma}^2p^2L)\right]\mathbb{E}\int_s^t|y(r)-Y(r)|^{p}dr
\\
&+2(pL+\bar{\sigma}^2p^2L)\mathbb{E}\int_s^t|y(r)-Y(r)|^{p-2}|Y(r)-Y_{\Delta}(r)|^2dr.
\end{aligned}
$$
With the help of the  H\"{o}lder inequality and Lemma
\ref{lemma:3.5}, together with $p\geq 2$, we have
$$
\begin{aligned}
&\mathbb{E}|y(t)-Y(t)|^p \\ \leq &
(4pL-4L+\frac{p}{2}+\bar{\sigma}^2(\frac{p}{2}+4p^2L-4pL))\mathbb{E}\int_s^t|y(r)-Y(r)|^{p}dr
\\
&+4(L+\bar{\sigma}^2pL)\mathbb{E}\int_s^t|Y(r)-Y_{\Delta}(r)|^pdr
\\
\leq
&(4pL-4L+\frac{p}{2}+\bar{\sigma}^2(\frac{p}{2}+4p^2L-4pL))\mathbb{E}\int_s^t|y(r)-Y(r)|^{p}dr
+4D_1(L+\bar{\sigma}^2pL)\Delta^{\frac{p}{2}}
\\
&\times
\left[(t-s)+2\left(\mathbb{E}|\eta(s)|^p+(\widehat{L}+\bar{\sigma}^2p\widehat{L})(t-s)\right)
e^{\frac{p+3p\widehat{L}-2\widehat{L}+\bar{\sigma}^2p(1-2\widehat{L}+3p\widehat{L})}{2}(t-s)}
\right].
\end{aligned}
$$
Applying the Gronwall inequality, we get
$$
\begin{aligned}
&\sup_{s \leq t \leq T}\mathbb{E}|y(t)-Y(t)|^p \\ \leq &
4D_1(L+\bar{\sigma}^2pL)\Delta^{\frac{p}{2}}
e^{\left(4pL-4L+\frac{p}{2}+\bar{\sigma}^2(\frac{p}{2}+4p^2L-4pL)\right)(T-s)}
\\
&\times
\left[(T-s)+2\left(\mathbb{E}|\eta(s)|^p+(\widehat{L}+\bar{\sigma}^2p\widehat{L})(T-s)\right)
e^{\frac{p+3p\widehat{L}-2\widehat{L}+\bar{\sigma}^2p(1-2\widehat{L}+3p\widehat{L})}{2}(T-s)}
\right].
\end{aligned}
$$
\end{proof}

The theorem below indicates that if $\Delta$ is sufficiently small,
then the EM method can reproduce the practical exponential
stability in $p$-th moment for the underlying $G$-SDE.

\begin{theorem} \label{th:7} Assume that the $G$-SDE \eqref{eq:2} is practically
exponential stable in $p$-th moment and that (H1) holds. For any
given $\delta \in(0,1)$, let $T=\left(\left\lfloor\frac{\ln
\left(2^{p-1} M_2 / \delta\right)}{\lambda_2
\Delta}\right\rfloor+1\right) \Delta.$ If $\Delta$ satisfies
$$
U(\Delta):=\left[\delta+
2^{p+2}D_1(L+\bar{\sigma}^2pL)e^{\left(p+4pL-8L+\bar{\sigma}^2(p+8p^2L-8pL)\right)T}
e^{\left(p+3p\widehat{L}-2\widehat{L}+\bar{\sigma}^2(p-2p\widehat{L}+3p^2\widehat{L})\right)T}\right]
\Delta^{\frac{p}{2}}<1,
$$
then the $G$-EMSDE \eqref{eq:4} is also practically exponential
stable in $p$-th moment.
\end{theorem}
\begin{proof}
Since $T$ is the multiple of $\Delta$, according to the flow
property of  the $G$-EMSDE \eqref{eq:4}, for arbitrary $i=0,1,2,
\ldots$, we can define $Y(t)=Y(t ; 0, \xi(0))=Y(t ; i T, Y(i T))$
and $y(t)=y(t ; i T, Y(i T))$. Furthermore, due to the autonomous
property of the $G$-SDE \eqref{eq:2}, the practical exponential
stability \eqref{def-pes-gsde} implies
$$
\mathbb{E}|y(t)|^p \leq M_2 \mathbb{E}|Y(i T)|^p
e^{-\lambda_2(t-iT)}+d_2, \quad \forall t \geq iT .
$$
Using the elementary inequality and Lemma \ref{lemma:3.5}, we have
\begin{eqnarray}\label{ineq:16}
&&\mathbb{E}|Y(t)|^p \nonumber \\&\leq &  2^{p-1} \mathbb{E}|y(t)|^p+2^{p-1} \mathbb{E}|y(t)-Y(t)|^p \nonumber \\
&\leq & 2^{p-1} M_2 \mathbb{E}|Y(iT)|^p e^{-\lambda_2(t-iT)} + 2^{p-1}d_2 \nonumber \\
 &&+2^{p+1}D_1(L+\bar{\sigma}^2pL)e^{\left(\frac{p}{2}+4pL-4L+\bar{\sigma}^2(\frac{p}{2}+4p^2L-4pL)\right)(t-iT)}
 \nonumber \\
&&\times
\left[(t-iT)+2\left(\mathbb{E}|Y(iT)|^p+(\widehat{L}+\bar{\sigma}^2p\widehat{L})(t-iT)\right)
e^{\frac{p+3p\widehat{L}-2\widehat{L}+\bar{\sigma}^2p(1-2\widehat{L}+3p\widehat{L})}{2}(t-iT)}
\right]\Delta^{\frac{p}{2}}\nonumber \\
&\leq & 2^{p-1} M_2 \mathbb{E}|Y(iT)|^p e^{-\lambda_2(t-iT)} + 2^{p-1}d_2 \nonumber \\
 &&+2^{p+1}D_1(L+\bar{\sigma}^2pL)e^{\left(\frac{p}{2}+4pL-4L+\bar{\sigma}^2(\frac{p}{2}+4p^2L-4pL)\right)(t-iT)}
 \nonumber \\
&&\times
\left[\left(2\mathbb{E}|Y(iT)|^p+(1+2\widehat{L}+2\bar{\sigma}^2p\widehat{L})(t-iT)\right)
e^{\frac{p+3p\widehat{L}-2\widehat{L}+\bar{\sigma}^2p(1-2\widehat{L}+3p\widehat{L})}{2}(t-iT)}
\right]\Delta^{\frac{p}{2}},
\end{eqnarray}
By the definition of $T$, we get $$ 2^{p-1} M_2 e^{-\lambda_2 T}
\leq \delta.
$$
This, together with \eqref{ineq:16}, yields that
\begin{eqnarray*}
\sup_{(i+1) T \leq t \leq(i+2) T} \mathbb{E}|Y(t)|^p \leq
U(\Delta)\mathbb{E}|Y(iT)|^p+d_4,
\end{eqnarray*}
where
$$
\begin{aligned}
d_4=& 2^{p-1}d_2
+2^{p+1}D_1T\tau^{\frac{p}{2}}(L+\bar{\sigma}^2pL)(1+2\widehat{L}+2\bar{\sigma}^2p\widehat{L})
\\
&\times e^{\left(p+4pL-8L+\bar{\sigma}^2(p+8p^2L-8pL)\right)T}
e^{\left(p+3p\widehat{L}-2\widehat{L}+\bar{\sigma}^2(p-2p\widehat{L}+3p^2\widehat{L})\right)T}.
\end{aligned}
$$
Since $U(\Delta)<1$, there is a number $\gamma_2>0$ that makes
$U(\Delta)=e^{-\gamma_2 T}$. Thus, by iteration, we can get that
$$
\begin{aligned}
\sup _{(i+1) T \leq t \leq(i+2)T} \mathbb{E}|Y(t)|^p & \leq e^{-\gamma_2 T} \mathbb{E}|Y(iT)|^p +d_4\\
& \leq e^{-\gamma_2 T} \sup _{iT \leq t \leq(i+1)T} \mathbb{E}|Y(t)|^p +d_4\\
& \leq e^{-\gamma_2(i+1)T} \sup _{0 \leq t \leq
T}\mathbb{E}|Y(t)|^p+d_4(e^{-\gamma_2i
T}+e^{-\gamma_2(i-1)T}+\cdots+e^{-\gamma_2T}+1).
\end{aligned}
$$
Using \eqref{ineq:16} with $i=0$, we have
$$
\begin{aligned}
\sup _{0 \leq t \leq T} \mathbb{E}|Y(t)|^p \leq &
2^{p-1} M_2 \mathbb{E}|\xi(0)|^p  + 2^{p-1}d_2 \nonumber \\
 &+2^{p+1}D_1(L+\bar{\sigma}^2pL)e^{\left(\frac{p}{2}+4pL-4L+\bar{\sigma}^2(\frac{p}{2}+4p^2L-4pL)\right)T}
 \nonumber \\
&\times
\left[\left(2\mathbb{E}|\xi(0)|^p+T(1+2L+2\bar{\sigma}^2pL)\right)
e^{\frac{p+3pL-2L+\bar{\sigma}^2p(1-2L+3pL)}{2}T}
\right]\Delta^{\frac{p}{2}}
\\ \leq& (2^{p-1} M_2+ e^{-\gamma_2
T})\mathbb{E}|\xi(0)|^p+d_4.
\end{aligned}
$$
Hence
$$
\begin{aligned}
\sup _{(i+1) T \leq t \leq(i+2) T} \mathbb{E}|Y(t)|^p \leq&
\left(2^{p-1} M_2+e^{-\gamma_2 T}\right) \mathbb{E}|\xi(0)|^p
e^{-\gamma_2(i+1) T}
 \\& +d_4(e^{-\gamma_2(i+1)
T}+e^{-\gamma_2i
T}+e^{-\gamma_2(i-1)T}+\cdots+e^{-\gamma_2T}+1)\\
\leq& \left(2^{p-1} M_2e^{\gamma_2 T}+1\right) \mathbb{E}|\xi(0)|^p
e^{-\gamma_2(i+2) T}
\\
& +d_4(e^{-\gamma_2(i+1) T}+e^{-\gamma_2i
T}+e^{-\gamma_2(i-1)T}+\cdots+e^{-\gamma_2T}+1).
\end{aligned}
$$
Consequently, for any $t \geq 0$, we can deduce that
$$
\mathbb{E}|Y(t)|^p \leq L_2\mathbb{E}|\xi(0)|^p e^{-\gamma_2t}+k_2,
$$
where $L_2=(2^{p-1} M_2e^{\gamma_2 T}+1)$ and
$k_2=\frac{d_4}{1-e^{-\gamma_2 T}}$, then we complete the proof.
\end{proof}

\subsection{$G$-EMSDDE  shares stability with $G$-EMSDE}

\begin{lemma} \label{lemma:4.8} Let $X$ and $Y$ be the solution of the $G$-EMSDDE
\eqref{eq:3} and the $G$-EMSDE \eqref{eq:4}, respectively. Assume
that $s \geq 0$ is a multiple of $\Delta$. Under (H1), for any $T>s
\geq 0$, it holds that
$$
\begin{aligned}
 \sup _{s \leq t \leq T} \mathbb{E} |X(t; s, \eta)-Y(t; s, \eta(s))|^p  \leq
d_5\tau e^{\frac{p+5pL-4L+\bar{\sigma}^2(p+5p^2L-4pL)}{2}(T-s)},
\end{aligned}
$$
where $d_5=2(L+\bar{\sigma}^2pL) \left[\tau^{\frac{p}{2}-1}\left(K_1
\|\eta\|_{\mathbb{E}}^{p}
e^{\frac{p+2p\widehat{L}+\bar{\sigma}^2(p+p^2\widehat{L}+p\widehat{L})}{2}(T-s)}+N_1(T-s)\right)
+ (K_2\|\eta\|_{\mathbb{E}}^{p}+N_2)\right]$.
\end{lemma}
\begin{proof}
 Write
$X(t ; s, \eta)=X(t), Y(t ; s, \eta(s))=Y(t)$, we have
$$
\begin{aligned}
X(t)-Y(t)= & \int_s^t\left(f\left(X_{\Delta}(r),
X_{\Delta}(r-\tau)\right)- f\left(Y_{\Delta}(r),
Y_{\Delta}(r)\right)\right) dr \\&
+\int_s^t\left(g\left(X_{\Delta}(r),
X_{\Delta}(r-\tau)\right)-g\left(Y_{\Delta}(r),
Y_{\Delta}(r)\right)\right) d\langle B \rangle(r)
 \\& +\int_s^t\left(h\left(X_{\Delta}(r),
X_{\Delta}(r-\tau)\right)-h\left(Y_{\Delta}(r),
Y_{\Delta}(r)\right)\right) dB(r) .
\end{aligned}
$$
In the same way as in the proof of Lemma \ref{lemma:3.6} we can show
that
$$
\begin{aligned}
&\mathbb{E}|X(t)-Y(t)|^p
\\
\leq &
\frac{p+p\bar{\sigma}^2}{2}\mathbb{E}\int_s^t|X(r)-Y(r)|^{p}dr +
\frac{pL+p^2L\bar{\sigma}^2}{2}
\mathbb{E}\int_s^t|X(t)-Y(t)|^{p-2}|X_{\Delta}(r)-Y_{\Delta}(r)|^2dr \\
&+ \frac{pL+p^2L\bar{\sigma}^2}{2}
\mathbb{E}\int_s^t|X(t)-Y(t)|^{p-2}|X_{\Delta}(r-\tau)-Y_{\Delta}(r)|^2dr.
\end{aligned}
$$
Applying the H\"{o}lder inequality, we can get that
$$
\begin{aligned}
&\mathbb{E}|X(t)-Y(t)|^p
\\
\leq &
\frac{p+5pL-10L+\bar{\sigma}^2(p+5p^2L-10pL)}{2}\mathbb{E}\int_s^t|X(r)-Y(r)|^{p}dr
\\
&+
3(L+\bar{\sigma}^2pL)\mathbb{E}\int_s^t|X_{\Delta}(r)-Y_{\Delta}(r)|^pdr
+
2(L+\bar{\sigma}^2pL)\mathbb{E}\int_s^t|X_{\Delta}(r)-X_{\Delta}(r-\tau)|^pdr
\end{aligned}
$$
By Lemma \ref{lemma:3.3}, for any $t_1\in[s,T]$, we derive that
$$
\begin{aligned}
&\sup_{s\leq t\leq t_1}\mathbb{E}|X(t)-Y(t)|^p
\\
\leq &
\frac{p+5pL-4L+\bar{\sigma}^2(p+5p^2L-4pL)}{2}\mathbb{E}\int_s^{t_1}\sup_{s\leq
u\leq r}|X(u)-Y(u)|^{p}dr
\\
& +
2(L+\bar{\sigma}^2pL)\mathbb{E}\int_s^{t_1}|X_{\Delta}(r)-X_{\Delta}(r-\tau)|^pdr
\\
\leq &
\frac{p+5pL-4L+\bar{\sigma}^2(p+5p^2L-4pL)}{2}\mathbb{E}\int_s^{t_1}\sup_{s\leq
u\leq r}|X(u)-Y(u)|^{p}dr
\\
& + 2(L+\bar{\sigma}^2pL) \left[\tau^{\frac{p}{2}}\left(K_1
\|\eta\|_{\mathbb{E}}^{p}
e^{\frac{p+2p\widehat{L}+\bar{\sigma}^2(p+p^2\widehat{L}+p\widehat{L})}{2}(t_1-s)}+N_1(s,
t_1)(t_1-s)\right)+ \tau (K_2\|\eta\|_{\mathbb{E}}^{p}+N_2)\right].
\end{aligned}
$$
Applying the Gronwall inequality, we obtain
$$
\begin{aligned}
 \sup_{s\leq t\leq T}\mathbb{E}|X(t)-Y(t)|^p \leq
d_5\tau e^{\frac{p+5pL-4L+\bar{\sigma}^2(p+5p^2L-4pL)}{2}(T-s)}.
\end{aligned}
$$
The proof is complete.
\end{proof}

 In the following, we show that the $G$-EMSDDE \eqref{eq:3}
shares the $p$-th moment practical exponential stability with the
$G$-EMSDE \eqref{eq:4}, provided that the delay $\tau$ is
sufficiently small.

\begin{theorem}\label{th:9} Assume that the $G$-EMSDE \eqref{eq:4} is
practically exponential stable in $p$-th moment and that  (H1)
holds. For any given $\delta \in(0,1)$, let
$$
T=\left(\left\lfloor\frac{\ln \left(2^{p-1} L_2 /
\delta\right)}{\gamma_2 \tau}\right\rfloor+2\right) \tau .
$$
If $\tau$ satisfies
$$
V(\tau):=\delta+2^{p}(L+\bar{\sigma}^2pL)\left(\tau^{\frac{p}{2}}K_1
e^{\left(p+2p\widehat{L}+\bar{\sigma}^2(p+p^2\widehat{L}+p\widehat{L})\right)T}+\tau
K_2\right)e^{\left(p+5pL-4L+\bar{\sigma}^2(p+5p^2L-4pL)\right)T}<1,
$$
then the $G$-EMSDDE \eqref{eq:3} is also practically exponential
stable in $p$-th moment.
\end{theorem}

\begin{proof}
Note that $T$ is the multiple of $\tau$ and hence of $\Delta$. By
the flow property of the $G$-EMSDDE \eqref{eq:3}, we can denote
$X(t)=X(t ; 0, \xi)=X\left(t ; i T, X_{i T}\right)$ for arbitrary
$i=0,1,2, \ldots$, where $\|X_{iT}\|_{\mathbb{E}}^{p}=\sup_{iT-\tau
\leq u \leq iT} \mathbb{E}|X(u)|^p$. Moreover, let $Y(t)=Y(t ; i T,
X(i T))$, due to the autonomous property of the $G$-SDE
\eqref{eq:4}, the practical exponential stability
\eqref{def-pes-EMgsde} implies
$$\mathbb{E}|Y(t)|^p \leq L_2 \mathbb{E}|X(iT)|^p e^{-\gamma_2 (t-iT)}+k_2, \ \forall t\geq iT.$$
For any $t \geq i T$, using the elementary inequality and Lemma
 \ref{lemma:4.8}, we have
\begin{eqnarray}\label{ineq:3.10}
&&\mathbb{E}|X(t)|^p\nonumber
\\
&\leq & 2^{p-1} \mathbb{E}|Y(t)|^p+
2^{p-1} \mathbb{E}|X(t)-Y(t)|^p \nonumber\\
&\leq & 2^{p-1} L_2 \mathbb{E}|X(iT)|^p e^{-\gamma_2 (t-iT)}+2^{p-1} k_2+2^{p}(L+\bar{\sigma}^2pL) \nonumber\\
&&\times \left[\tau^{\frac{p}{2}}\left(K_1
\|X_{iT}\|_{\mathbb{E}}^{p}
e^{\frac{p+2p\widehat{L}+\bar{\sigma}^2(p+p^2\widehat{L}+p\widehat{L})}{2}(t-iT)}+N_1(iT,
t)(t-iT)\right) + \tau (K_2\|X_{iT}\|_{\mathbb{E}}^{p}+N_2)\right]
\nonumber\\
&&\times e^{\frac{p+5pL-4L+\bar{\sigma}^2(p+5p^2L-4pL)}{2}(t-iT)}.
\end{eqnarray}
 The definition of $T$ shows directly that
$$
2^{p-1} L_2 e^{-\gamma_2(T-\tau)} \leq \delta .
$$Thus
$$
\begin{aligned}
\sup_{(i+1) T-\tau \leq t \leq(i+2) T-\tau} \mathbb{E}|X(t)|^p \leq
& V(\tau)\|X_{iT}\|_{\mathbb{E}}^{p}+d_6,
\end{aligned}
$$
where $$d_6=2^{p-1}
k_2+2^{p}(L+\bar{\sigma}^2pL)(2\tau^{\frac{p}{2}}TN_1(0,T)+\tau N_2)
e^{(p+5pL-4L+\bar{\sigma}^2(p+5p^2L-4pL))T}.$$

Since $V(\tau)<1$, there exists a constant $\gamma_1>0$ such that
$V(\tau)=e^{-\gamma_1 T}$, then
$$
\sup _{(i+1)T-\tau \leq t \leq(i+2)T-\tau} \mathbb{E}|X(t)|^P \leq
e^{-\gamma_1 T} \|X_{iT}\|_{\mathbb{E}}^{p}+d_6.
$$
By iteration, we can get that
$$
\begin{aligned}
&\sup _{(i+1)T-\tau \leq t \leq(i+2)T-\tau} \mathbb{E}|X(t)|^P \\
& \leq e^{-\gamma_1 T} \sup_{iT-\tau \leq t \leq(i+1)T-\tau} \mathbb{E}|X(t)|^P+d_6 \\
& \leq e^{-\gamma_1 iT} \sup_{T-\tau \leq t \leq 2 T-\tau}
\mathbb{E}|X(t)|^P+d_6(e^{-\gamma_1(i-1)T}+e^{-\gamma_1(i-2)T}+\cdots+e^{-\gamma_1T}+1)\\
& \leq  e^{-\gamma_1 (i+1)T}
\|\xi\|_{\mathbb{E}}^{p}+d_6(e^{-\gamma_1iT}+e^{-\gamma_1(i-1)T}+\cdots+e^{-\gamma_1T}+1)
\\
& \leq  e^{\gamma_1T}
\|\xi\|_{\mathbb{E}}^{p}e^{-\gamma_1((i+2)T-\tau)}+\frac{d_6}{1-e^{-\gamma_1T}}.
\end{aligned}
$$
Using \eqref{ineq:3.10}, for $i=0$, we obtain that
$$
\begin{aligned}
& \sup_{0 \leq t \leq T-\tau}\mathbb{E}|X(t)|^p
\\
\leq & 2^{p-1} L_2 \mathbb{E}|\xi(0)|^p +2^{p-1} k_2 +2^{p} (L+\bar{\sigma}^2pL) \\
&\times  \left[\tau^{\frac{p}{2}}\left(K_1 \|\xi\|_\mathbb{E}^p
\times
e^{\frac{p+2p\widehat{L}+\bar{\sigma}^2(p+p^2\widehat{L}+p\widehat{L})}{2}(T-\tau)}+N_1(0,
T-\tau)(T-\tau)\right) + \tau (K_2\|\xi\|_\mathbb{E}^p+N_2)\right]
\\&\times e^{\frac{p+5pL-4L+\bar{\sigma}^2(p+5p^2L-4pL)}{2}(T-\tau)}
\\ \leq & 2^{p-1} L_2 \mathbb{E}|\xi(0)|^p +
V(\tau)\|\xi\|_\mathbb{E}^p + d_6
\\ \leq & (2^{p-1}L_2e^{\gamma_1 T} +1)\|\xi\|_\mathbb{E}^p e^{-\gamma_1 T}+
d_6.
\end{aligned}
$$
Hence
$$
\mathbb{E}|X(t)|^p \leq L_1\|\xi\|_\mathbb{E}^p e^{-\gamma_1
T}+\frac{d_6}{1-e^{-\gamma_1t}}, \ \forall t \geq 0,
$$
where $L_1=\left(2^{p-1} L_2+1\right) e^{\gamma_1 T}$. The proof is
complete.
\end{proof}

\subsection{$G$-SDDE  shares stability with
$G$-EMSDDE}
\begin{lemma}\label{lemma:4.9} Let $X$ be the solution of the $G$-EMSDDE \eqref{eq:3}
and let $s \geq 0$ be a multiple of $\Delta$. Under  (H1), it holds
that
$$
\sup _{s \leq t \leq T} \mathbb{E}\left|X(t ; s, \eta)-X_{\Delta}(t
; s, \eta)\right|^p \leq d_7\Delta^{\frac{p}{2}}, \quad \forall T>s,
$$
where
$$
\begin{aligned}d_7=&3^{\frac{3p}{2}-2}\widehat{L}^{\frac{p}{2}}\left[\tau^{\frac{p}{2}}+\bar{\sigma}^{2p}\tau^{\frac{p}{2}}
+\left(\frac{p^{p+1}}{2(p-1)^{p-1}}\right)^{\frac{p}{2}}\bar{\sigma}^{p}\right]\Delta^{\frac{p}{2}}
\\
 &
\times\left[1+2\left((\widehat{L}+\bar{\sigma}^2p\widehat{L})(T-s)+
(1+\widehat{L}\tau+\bar{\sigma}^2p\widehat{L}\tau)\|\eta\|_\mathbb{E}^p\right)
e^{\frac{p+3p\widehat{L}-2\widehat{L}+\bar{\sigma}^2(p+2p^2\widehat{L}-p\widehat{L})}{2}(T-s)}\right].
\end{aligned}
$$
\end{lemma}
\begin{proof}
  Write $X(t ; s, \eta)=X(t), X_{\Delta}(t ; s, \eta)=X_{\Delta}(t)$. Let $n_s \in
\mathbb{N}$ such that $s=n_s \Delta$, then for arbitrary $t \geq s$,
there exists an $n \in \mathbb{N}$ such that $t \in$ $[(n_s+n)
\Delta, (n_s+n+1) \Delta)$. By \eqref{eq:d1} and \eqref{eq:3}, we
have
$$
\begin{aligned}
X(t)-X_{\Delta}(t) =& X(t)-X_{n+n_s} \\
=& \int_{s+n \Delta}^t f\left(X_{\Delta}(r),
X_{\Delta}(r-\tau)\right) dr+\int_{s+n \Delta}^t
g\left(X_{\Delta}(r), X_{\Delta}(r-\tau)\right) d\langle B\rangle(r)
\\
& +\int_{s+n \Delta}^t h\left(X_{\Delta}(r),
X_{\Delta}(r-\tau)\right) dB(r) .
\end{aligned}
$$
By the similar procedure in the proof of Lemma \ref{lemma:3.5}, one
can arrive at
$$
\begin{aligned}
\mathbb{E}\left|X(t)-X_{\Delta}(t)\right|^p \leq & 3^{p-1}
\Delta^{p-1} \mathbb{E}\int_{s+n \Delta}^t
\left|f\left(X_{\Delta}(r),
X_{\Delta}(r-\tau)\right)\right|^p dr \\
& +3^{p-1}\bar{\sigma}^{2p}\Delta^{p-1} \mathbb{E}\int_{s+n
\Delta}^t \left|g\left(X_{\Delta}(r),
X_{\Delta}(r-\tau)\right)\right|^p dr \\
&+3^{p-1}\left(\frac{p^{p+1}}{2(p-1)^{p-1}}\right)^{\frac{p}{2}}\bar{\sigma}^{p}\Delta^{\frac{p}{2}-1}
\mathbb{E}\int_{s+n \Delta}^t \left|h\left(X_{\Delta}(r),
X_{\Delta}(r-\tau)\right)\right|^pdr.
\end{aligned}
$$
Using \eqref{pthlg} and Lemma \ref{lemma:3.1}, we get
$$
\begin{aligned}
&\mathbb{E}\left|X(t)-X_{\Delta}(t)\right|^p
\\ \leq &
3^{\frac{3p}{2}-2}\widehat{L}^{\frac{p}{2}}\left[\Delta^{p-1}+\bar{\sigma}^{2p}\Delta^{p-1}
+\left(\frac{p^{p+1}}{2(p-1)^{p-1}}\right)^{\frac{p}{2}}\bar{\sigma}^{p}\Delta^{\frac{p}{2}-1}\right]
\\
 &\times\mathbb{E}\int_{s+n \Delta}^t
\left[1+2\left((\widehat{L}+\bar{\sigma}^2p\widehat{L})(r-s)+
(1+\widehat{L}\tau+\bar{\sigma}^2p\widehat{L}\tau)\|\eta\|_\mathbb{E}^p\right)
e^{\frac{p+3p\widehat{L}-2\widehat{L}+\bar{\sigma}^2(p+2p^2\widehat{L}-p\widehat{L})}{2}(r-s)}\right]
dr.
\end{aligned}
$$
Thus
$$
\begin{aligned}
&\sup _{s \leq t \leq T}\mathbb{E}\left|X(t)-X_{\Delta}(t)\right|^p
\\ \leq & 3^{\frac{3p}{2}-2}\widehat{L}^{\frac{p}{2}}\left[\tau^{\frac{p}{2}}+\bar{\sigma}^{2p}\tau^{\frac{p}{2}}
+\left(\frac{p^{p+1}}{2(p-1)^{p-1}}\right)^{\frac{p}{2}}\bar{\sigma}^{p}\right]\Delta^{\frac{p}{2}}
\\
 &
\times\left[1+2\left((\widehat{L}+\bar{\sigma}^2p\widehat{L})(T-s)+
(1+\widehat{L}\tau+\bar{\sigma}^2p\widehat{L}\tau)\|\eta\|_\mathbb{E}^p\right)
e^{\frac{p+3p\widehat{L}-2\widehat{L}+\bar{\sigma}^2(p+2p^2\widehat{L}-p\widehat{L})}{2}(T-s)}\right].
\end{aligned}
$$
The proof is completed.
\end{proof}
\begin{lemma} \label{lemma:4.10} Let $x$ and $X$ be the solution of the $G$-SDDE \eqref{eq:1} and the
$G$-EMSDDE \eqref{eq:3}, respectively. Assume that $s \geq 0$ is a
multiple of $\Delta$. Under (H1), it holds that
$$
\begin{aligned}
&\sup_{s \leq t \leq T} \mathbb{E}|x(t)-X(t)|^p \\ \leq
&3^{\frac{3p}{2}-2}\widehat{L}^{\frac{p}{2}}(4L+4\bar{\sigma}^2pL)
\left[\tau^{\frac{p}{2}}+\bar{\sigma}^{2p}\tau^{\frac{p}{2}}
+\left(\frac{p^{p+1}}{2(p-1)^{p-1}}\right)^{\frac{p}{2}}\bar{\sigma}^{p}\right]\Delta^{\frac{p}{2}}
\\
 &
\times\left[1+2\left((\widehat{L}+\bar{\sigma}^2p\widehat{L})(T-s)+
(1+\widehat{L}\tau+\bar{\sigma}^2p\widehat{L}\tau)\|\eta\|_\mathbb{E}^p\right)
e^{\frac{p+3p\widehat{L}-2\widehat{L}+\bar{\sigma}^2(p+2p^2\widehat{L}-p\widehat{L})}{2}(T-s)}\right]
\\
 & \times e^{\frac{p+8pL-8L+\bar{\sigma}^2(p+8p^2L-8pL)}{2}(T-s)}, \  \forall
T>s.
\end{aligned}
$$
\end{lemma}
\begin{proof} Let $x(t)=x(t ; s, \eta), X(t)=X(t ; s, \eta)$, then
$$
\begin{aligned}
x(t)-X(t)= & \int_s^t\left(f(x(r), x(r-\tau))-f(X_{\Delta}(r), X_{\Delta}(r-\tau))\right) dr \\
& +\int_s^t\left(g(x(r), x(r-\tau))-g(X_{\Delta}(r),
X_{\Delta}(r-\tau))\right) d\langle B \rangle (r)
\\
& +\int_s^t\left(h(x(r), x(r-\tau))-h(X_{\Delta}(r),
X_{\Delta}(r-\tau))\right) dB(r)  .
\end{aligned}
$$
Using the $G$-It\^{o}'s formula, the elementary inequality,
 (H1), Lemmas \eqref{lemma:2.1} and \eqref{lemma:2.2},  we
have
$$\begin{aligned}
&\mathbb{E}|x(t)-X(t)|^p\\
\leq &
\frac{p(1+\bar{\sigma}^2)}{2}\mathbb{E}\int_s^t|x(r)-X(r)|^{p}dr
+\left(\frac{pL}{2}+\bar{\sigma}^2\frac{p^2L}{2}\right)\mathbb{E}\int_s^t|x(r)-X(r)|^{p-2}|x(r)-X_{\Delta}(r)|^2dr
\\
&
+\left(\frac{pL}{2}+\bar{\sigma}^2\frac{p^2L}{2}\right)
\mathbb{E}\int_s^t|x(r)-X(r)|^{p-2}|x(r-\tau)-X_{\Delta}(r-\tau)|^2dr.
\end{aligned}$$
With the help of the plus-minus technique, the elementary inequality
and the H\"{o}lder inequality, one can arrive at
$$\begin{aligned}
&\mathbb{E}|x(t)-X(t)|^p
\\ \leq &
\frac{p+\bar{\sigma}^2p}{2}\mathbb{E}\int_s^t|x(r)-X(r)|^{p}dr
\\
& +(pL+\bar{\sigma}^2p^2L) \mathbb{E}\int_s^t|x(r)-X(r)|^{p-2}
(|x(r)-X(r)|^2+|X(r)-X_{\Delta}(r)|^2)dr
\\
& +(pL+\bar{\sigma}^2p^2L)
\mathbb{E}\int_s^t|x(r)-X(r)|^{p-2}(|x(r-\tau)-X(r-\tau)|^2+|X(r-\tau)-X_{\Delta}(r-\tau)|^2)dr
\\ & \leq
\frac{p+8pL-12L+\bar{\sigma}^2(p+8p^2L-12pL)}{2}\mathbb{E}\int_s^t|x(r)-X(r)|^{p}dr
\\
& +(2L+2\bar{\sigma}^2pL)\mathbb{E}\int_s^t|X(r)-X_{\Delta}(r)|^pdr
+(2L+2\bar{\sigma}^2pL)\mathbb{E}\int_s^t|x(r-\tau)-X(r-\tau)|^pdr
\\
& +(2L+2\bar{\sigma}^2pL)
\mathbb{E}\int_s^t|X(r-\tau)-X_{\Delta}(r-\tau)|^pdr.
\end{aligned}$$
Then, for any $t_1 \in[s, T]$, using Lemma \ref{lemma:4.9} , one can
obtain that
$$\begin{aligned}
&\sup _{s \leq t \leq {t_1}} \mathbb{E}|x(t)-X(t)|^p
\\ \leq &
\frac{p+8pL-8L+\bar{\sigma}^2(p+8p^2L-8pL)}{2}\mathbb{E}\int_s^{t_1}|x(r)-X(r)|^{p}dr
\\
& +(4L+4\bar{\sigma}^2pL) \mathbb{E}\int_s^{t_1}(\sup _{s \leq u\leq
r}|X(u)-X_{\Delta}(u)|^p)dr
\\ \leq &
\frac{p+8pL-8L+\bar{\sigma}^2(p+8p^2L-8pL)}{2}\mathbb{E}\int_s^{t_1}|x(r)-X(r)|^{p}dr
\\
& +3^{\frac{3p}{2}-2}\widehat{L}^{\frac{p}{2}}(4L+4\bar{\sigma}^2pL)
\left[\tau^{\frac{p}{2}}+\bar{\sigma}^{2p}\tau^{\frac{p}{2}}
+\left(\frac{p^{p+1}}{2(p-1)^{p-1}}\right)^{\frac{p}{2}}\bar{\sigma}^{p}\right]\Delta^{\frac{p}{2}}
\\
 &
\times\left[1+2\left((\widehat{L}+\bar{\sigma}^2p\widehat{L})(t_1-s)+
(1+\widehat{L}\tau+\bar{\sigma}^2p\widehat{L}\tau)\|\eta\|_\mathbb{E}^p\right)
e^{\frac{p+3p\widehat{L}-2\widehat{L}+\bar{\sigma}^2(p+2p^2\widehat{L}-p\widehat{L})}{2}(t_1-s)}\right],
\end{aligned}$$
whence the desired assertion follows from the Gronwall inequality.
\end{proof}

The theorem below shows that $G$-SDDE can preserve the practical
exponential stability of the corresponding $G$-EMSDDE  providing
that the step size $\Delta$ is small enough.

\begin{theorem}\label{th:12} Let (H1) holds and assume that the $G$-EMSDDE \eqref{eq:3} is practically exponential
stable in $p$-th moment. For any given $\delta \in(0,1)$, let
$$
T=\left(\left\lfloor\frac{\ln \left(2^{p-1} L_1 /
\delta\right)}{\gamma_1 \tau}\right\rfloor+3\right) \tau .
$$
If $\Delta>0$ satisfies
$$\begin{aligned}
W(\Delta):=&\delta+2^{p+2}
3^{\frac{3p}{2}-2}\widehat{L}^{\frac{p}{2}}(L+\bar{\sigma}^2pL)\left[\tau^{\frac{p}{2}}+\bar{\sigma}^{2p}\tau^{\frac{p}{2}}
+\left(\frac{p^{p+1}}{2(p-1)^{p-1}}\right)^{\frac{p}{2}}\bar{\sigma}^{p}\right]
(1+\widehat{L}\tau+\bar{\sigma}^2p\widehat{L}\tau)
\\ &\times
e^{p+3p\widehat{L}-2\widehat{L}+\bar{\sigma}^2(p+2p^2\widehat{L}-p\widehat{L})(T-\tau)}
e^{p+8pL-8L+\bar{\sigma}^2(p+8p^2L-8pL)(T-\tau)}\Delta^{\frac{p}{2}}
<1,
\end{aligned}$$
then the $G$-SDDE \eqref{eq:1} is also practically exponential
stable in $p$-th moment.
\end{theorem}
\begin{proof}
For any initial value $\xi$, due to the flow property of the
$G$-SDDE \eqref{eq:1}, we write $x\left(t ; \tau, x_\tau\right)=x(t
; 0, \xi)=x(t)$ and  $X\left(t ; \tau, x_\tau\right)=X(t)$. The
autonomous property and  the practical exponential stability
\eqref{def-pes-EMgsdde} of the $G$-EMSDDE \eqref{eq:3} implies that
$$
\mathbb{E}|X(t)|^p \leq L_1\left\|x_\tau\right\|_\mathbb{E}^p
e^{-\gamma_1(t-\tau)}+k_1, \quad \forall t \geq \tau.
$$
Using Lemma \ref{lemma:4.10}, we get
\begin{eqnarray}\label{ineq:3.11}
\mathbb{E} |x(t)|^p &\leq&  2^{p-1} \mathbb{E}|X(t)|^p+2^{p-1} \mathbb{E}|x(t)-X(t)|^p \nonumber \\
&\leq & 2^{p-1} L_1\left\|x_\tau\right\|_\mathbb{E}^p
e^{-\gamma_1(t-\tau)}+2^{p-1}k_1
\nonumber\\
 && +2^{p+1}
3^{\frac{3p}{2}-2}\widehat{L}^{\frac{p}{2}}(L+\bar{\sigma}^2pL)\left[\tau^{\frac{p}{2}}+\bar{\sigma}^{2p}\tau^{\frac{p}{2}}
+\left(\frac{p^{p+1}}{2(p-1)^{p-1}}\right)^{\frac{p}{2}}\bar{\sigma}^{p}\right]\Delta^{\frac{p}{2}}
\nonumber\\ &&
\times\left[1+2\left((\widehat{L}+\bar{\sigma}^2p\widehat{L})(t-\tau)+
(1+\widehat{L}\tau+\bar{\sigma}^2p\widehat{L}\tau)\|x_\tau\|_\mathbb{E}^p\right)
e^{\frac{p+3p\widehat{L}-2\widehat{L}+\bar{\sigma}^2(p+2p^2\widehat{L}-p\widehat{L})}{2}(t-\tau)}\right]
\nonumber\\&& \times
e^{\frac{p+8pL-8L+\bar{\sigma}^2(p+8p^2L-8pL)}{2}(t-\tau)}.
\end{eqnarray}
 According to the definition of $T$, we have
$ 2^{p-1} L_1 e^{-\gamma_1(T-2\tau)} \leq \delta \ \ \text{and} \
T\geq 2\tau. $ Hence, by a straightforward computation, we deduced
from  \eqref{ineq:3.11} that
$$
\begin{aligned}
 \sup_{T-\tau \leq t \leq 2T-\tau} \mathbb{E}|x(t)|^p
 \leq  W(\Delta)\left\|x_\tau\right\|_\mathbb{E}^p+d_8,
\end{aligned}
$$
where
$$
\begin{aligned}
d_8=& 2^{p-1}k_1 + 2^{p+1}
3^{\frac{3p}{2}-2}\widehat{L}^{\frac{p}{2}}(L+\bar{\sigma}^2pL)\left[\tau^{\frac{p}{2}}+\bar{\sigma}^{2p}\tau^{\frac{p}{2}}
+\left(\frac{p^{p+1}}{2(p-1)^{p-1}}\right)^{\frac{p}{2}}\bar{\sigma}^{p}\right]\Delta^{\frac{p}{2}}
\\ &
\times\left[1+4(\widehat{L}+\bar{\sigma}^2p\widehat{L})(T-\tau)
e^{[p+3p\widehat{L}-2\widehat{L}+\bar{\sigma}^2(p+2p^2\widehat{L}-p\widehat{L})](T-\tau)}\right]
e^{[p+8pL-8L+\bar{\sigma}^2(p+8p^2L-8pL)](T-\tau)}.
\end{aligned}
$$
Since $W(\Delta)<1$, choosing $\lambda_1>0$ such that
$W(\Delta)=e^{-\lambda_1 T}$, we have
$$
\begin{aligned}
 \sup_{T-\tau \leq t \leq 2T-\tau} \mathbb{E}|x(t)|^p  \leq
e^{-\lambda_1 T}\|x_\tau\|_\mathbb{E}^p+d_8.
\end{aligned}
$$
Recalling that $T$ is a multiple of $\tau$, repeating this
procedure, we get that
\begin{eqnarray}\label{ineq:4.11}
\sup_{(i+1) T-\tau \leq t \leq(i+2) T-\tau} \mathbb{E}|x(t)|^p &\leq
& e^{-(i+1) \lambda_1 T}\|x_\tau\|_\mathbb{E}^p+d_8(e^{-\gamma_1i
T}+e^{-\gamma_1(i-1)T}+\cdots+e^{-\gamma_1T}+1) \nonumber\\ &\leq
&e^{\lambda_1(T-\tau)}\|x_\tau\|_\mathbb{E}^pe^{-\lambda_1t}+\frac{d_8}{1-e^{-\gamma_1T}},
\ \forall i \geq 0.
\end{eqnarray}
By \eqref{ineq:3.11}, we  can deduce that
$$
\begin{aligned}
\sup_{\tau \leq t \leq T-\tau} \mathbb{E}|x(t)|^p \leq  (2^{p-1}
L_1+W(\Delta))\|x_\tau\|_\mathbb{E}^p+d_8,
\end{aligned}
$$
which gives
\begin{eqnarray}\label{ineq:4.12}
\sup _{0 \leq t \leq T-\tau} \mathbb{E}\|x(t)\|^p \leq (1+2^{p-1}
L_1)e^{\lambda_1(T-\tau)}\|x_\tau\|_\mathbb{E}^pe^{-\lambda_1t}+d_8.
\end{eqnarray}
Furthermore, using Lemma \ref{lemma:3.1} gives
\begin{eqnarray}\label{ineq:4.13}
\|x_\tau\|_\mathbb{E}^p \leq
 \left[(\widehat{L}+\bar{\sigma}^2p\widehat{L})\tau+
(1+\widehat{L}\tau+\bar{\sigma}^2p\widehat{L}\tau)\|\xi\|_{\mathbb{E}}^{p}\right]
e^{\frac{p+3p\widehat{L}-2\widehat{L}+\bar{\sigma}^2(p+2p^2\widehat{L}-p\widehat{L})}{2}\tau}.
\end{eqnarray}
Then, combining \eqref{ineq:4.11} and \eqref{ineq:4.12} with
\eqref{ineq:4.13}, we obtain that
$$
\begin{aligned}
\mathbb{E}|x(t)|^p \leq M_1\|\xi\|_{\mathbb{E}}^{p} +d_1, \ \forall
t\geq 0.
\end{aligned}
$$
where $M_1=(1+2^{p-1} L_1)
(1+\widehat{L}\tau+\bar{\sigma}^2p\widehat{L}\tau)
e^{\frac{p+3p\widehat{L}-2\widehat{L}+\bar{\sigma}^2(p+2p^2\widehat{L}-p\widehat{L})+2\lambda_1}{2}(T-\tau)}$
 and $d_1=(1+2^{p-1} L_1)
(\widehat{L}\tau+\bar{\sigma}^2p\widehat{L}\tau)
e^{\frac{p+3p\widehat{L}-2\widehat{L}+\bar{\sigma}^2(p+2p^2\widehat{L}-p\widehat{L})+2\lambda_1}{2}(T-\tau)}
+\frac{d_8}{1-e^{-\gamma_1T}}$. The proof is complete.
\end{proof}
\section{Conclusion}
Combining Theorems \ref{th:4}, \ref{th:7}, \ref{th:9} and
\ref{th:12}, we can give a summary of the relationships among
$G$-SDDEs, $G$-SDEs and their EM methods as the following theorem.
\begin{theorem}
Let (H1) holds, if either of the $G$-SDDE \eqref{eq:1}, $G$-SDE
\eqref{eq:2}, $G$-EMSDDE \eqref{eq:3} and $G$-EMSDE \eqref{eq:4} is
 practically exponential stable in $p$-th moment, then also the
 other three are practically exponential stable in $p$-th moment,
 provided that the delay $\tau$ and the step size $\Delta$ are small
 enough.
\end{theorem}
So far as we know, related works on the topic of practical stability
equivalence are very few. In this paper, enlightened by the recent
work \cite{zhsongLiu23} due to Zhang et al., we have investigated
the equivalence of practical exponential stability among $G$-SDDEs,
$G$-SDEs and their EM methods under the Lipschitz assumption. It
should be pointed out that the model and computations are nontrivial
due to the uncertainty of $G$-Brownian motion. We hope that the
results of this paper will play a fundamental role in future
research.

\label{lastpage-01}

\begin{thebibliography}{99}

\footnotesize


\bibitem{Bao21} Z. Bao, J. Tang, Y. Shen, Equivalence of $p$-th moment stability between stochastic
differential delay equations and their numerical methods, Stat.
Probabil. Lett. 168 (2021) 108952.


\bibitem{Caraballo23} T. Caraballo,  F. Ezzine, M. A.
Hammami., Practical stability with respect to a part of the
variables of stochastic differential equations driven by
$G$-Brownian motion, J. Dyn. Control Syst. 29 (2023) 1-19.

\bibitem{Caraballo14}  T. Caraballo, M. A. Hammami, L. Mchiri, Practical asymptotic
stability of nonlinear stochastic evolution equations. Stoch. Anal.
Appl. 32 (2014) 77-87.

\bibitem{Caraballo17} T. Caraballo, M. A. Hammami, L. Mchiri, Practical exponential
stability of impulsive stochastic functional differential equations,
Syst. Cont. Lett. 109 (2017) 43-48.

\bibitem{DFFM19}  S. Deng, C. Fei, W. Fei, X. Mao, Stability equivalence between the stochastic
differential delay equations driven by $G$-Brownian motion and the
Euler-Maruyama method,  Appl. Math. Lett. 96 (2019) 138-146.

\bibitem {DHP11} L. Denis, M. Hu, and S. Peng,  Function spaces and capacity related to a sublinear expectation:
application to $G$-Brownian motion pathes, Potential Anal.  34
(2011) 139-161.


\bibitem{Dlala07}M. Dlala, M.A. Hammami, Uniform exponential practical stability of
impulsive perturbed systems, J. Dyn. Control Syst. 13 (2007)
373-386.

\bibitem{FFY19}  C. Fei, W. Fei, L. Yan, Existence and Stability of
Solutions to Highly Nonlinear Stochastic Differential Delay
Equations Driven by $G$-Brownian Motion, Appl. Math. J. Chinese
Univ. 34 (2019) 184-204.

\bibitem{FFMY22} C. Fei, W. Fei, X. Mao, L. Yan, Delay-dependent Asymptotic
Stability of Highly Nonlinear Stochastic Differential Delay
Equations Driven by G-Brownian Motion, J. Franklin I. 359 (2022)
4366-4392.

\bibitem{FengLiu94} Z. Feng, Y. Liu, Practical stability analysis for stochastic
time-delay systems, Appl. Math. Comput. 63 (1994) 265-279.

\bibitem{Gao09} F. Gao, Pathwise properties and homeomorphic flows for stochastic
differential equations driven by $G$-Brownian motion, Stochastic
Process. Appl. 119 (2009) 3356-3382.

\bibitem{HuPeng09} M. Hu,  S. Peng, On the representation theorem of $G$-expectations and
 paths of $G$-Brownian motion, Acta Math. Appl.
Sin. Engl. Ser. 25(2009) 539-546.


\bibitem{Higham02} D.J. Higham, X. Mao, A.M. Stuart,  Strong convergence of
Euler-type methods for nonlinear stochastic differential equations,
SIAM J. Numer. Anal. 40 (2002) 1041-1063.

\bibitem{Higham03} D.J. Higham, X. Mao, A.M. Stuart, Exponential mean-square
stability of numerical solutions to stochastic differential
equations. LMS J. Comput. Math. 6 (2003) 297-313.

\bibitem{Hasm12} R. Hasminskii, Stochastic Stability of Differential
Equations, Second edition, Springer-Verlag, Berlin, Heidel-berg,
2012.

\bibitem{Hu14} L. Hu, Y. Ren, T. Xu, $p$-Moment stability of solutions to stochastic di?erential equations driven by G-Brownian motion,
Appl. Math. Comput. 230 (2014) 231-237.

\bibitem{Hu20} L. Hu, Y. Ren, R. Sakthivel, Stability of square-mean almost automorphic mild solutions to impulsive stochastic differential equations
driven by $G$-Brownian motion, Internat. J. Control 93 (2020)
3016-3025.

\bibitem{Hu96} Y. Hu, Semi-implicit Euler-Maruyama scheme for stiff stochastic
equations, in: H. Koerezlioglu (Ed.), Stochastic Analysis and
Related Topics V: The Silvri Workshop, Progress in Probability, vol.
38, Birkhauser, Boston, 1996, pp. 183-202.

\bibitem{Kloeden92} P.E. Kloeden, E. Platen, Numerical Solution of Stochastic Differential Equations,
 Applications of Mathematics, Springer, Berlin, 1992.

\bibitem{LLM90} V. Lakshmikantham, S. Leela and A. Martyuk, Practical stability of
nonlinear systems, World Scientific Publishing Co. Pte. Ltd.,
Singapore, 1990.

\bibitem{LiYang18} G. Li, Q. Yang,  Convergence and asymptotical stability of numerical
solutions for neutral stochastic delay differential equations driven
by $G$-Brownian motion, Comput. Appl. Math., 37 (2018) 4301-4320.

\bibitem{LLL16} X. Li,  X. Lin, Y. Lin, Lyapunov-type conditions and
stochastic differential equations driven by $G$-Brownian motion, J.
Math. Anal. Appl., 439 (2016) 235-255.

\bibitem {LiP11} X. Li, S. Peng, Stopping times and related It\^{o}'s calculus with $G$-Brownian motion,
 Stochastic Process. Appl., 121 (2011) 1492-1508.

\bibitem{LMY19} X. Li, X. Mao, G. Yin, Explicit numerical approximations for
stochastic differential equations in finite and infinite horizons:
truncation methods, convergence in $p$th moment and stability, IMA J
Numer. Anal., 39 (2019) 847-892.

\bibitem{LiuLu24} C. Liu, W. Lu, Convergence of the
Euler-Maruyama method for Stochastic Differential Equations driven
by $G$-Brownian motion, preprint, 2024.

\bibitem{LLD2018} L. Liu,, M. Li, F. Deng, Stability equivalence between the
neutral delayed stochastic differential equations and the
Euler-Maruyama numerical scheme. Appl. Numer. Math. 127 (2018)
370-386.

\bibitem{Mao2007} X. Mao, Exponential stability of equidistant Euler-Maruyama
approximations of stochastic differential delay equations. J.
Comput. Appl. Math. 200 (2007) 297-316.

\bibitem{Mao07} X. Mao, Stochastic Differential Equations and Their
Applications, 2nd Edition, Chichester, Horwood Pub, 2007.

\bibitem{MaoYuan06} X. Mao, C. Yuan, Stochastic Differential Equations with Markovian
Swtching, Imperial College Press, London, 2006.

\bibitem{Mao2015} X. Mao, Almost sure exponential stability in the numerical
simulation of stochastic differential equations. SIAM J. Numer.
Anal. 53 (2015) 370-389.

\bibitem{Peng07} S. Peng, $G$-expectation, $G$-Brownian motion and related stochastic
calculus of It\^{o} type. In: Stochastic analysis and applications,
Abel Symp., Vol.2, ed. by F.E., Benth, G. Di Nunno, T. Lindstrom, B.
{\O}ksendal, T. Zhang, Springer-Verlag, Berlin, 2007, 541-567


\bibitem{Peng08} S. Peng,  Multi-dimensional $G$-Brownian motion and related stochastic calculus
under $G$-expectation, Stochastic Process. Appl.,  118 (2008)
2223-2253.

\bibitem{Peng10} S. Peng, Nonlinear expectations and stochastic calculus under uncertainty with robust
CLT and G-Brownian motion, Springer, Berlin, Heidelberg, 2019.

\bibitem{ren17} Y. Ren, X. Jia, R. Sakthivel. The $p$-th moment stability of
solutions to impulsive stochastic differential equations driven by
G-Brownian motion,Applicable Analysis, 96 (2017) 988-1003.

\bibitem{ren18}
Y Ren, W Yin, R Sakthivel. Stabilization of stochastic differential
equations driven by G-Brownian motion with feedback control based on
discrete-time state observation, Automatica, 95 (2018) 146-151.

\bibitem{ren18} Y. Ren, W. Yin, D. Zhu, Exponential stability of SDEs driven by
G-Brownian motion with delayed impulsive effects: average impulsive
interval approach, Discrete Contin. Dyn. Syst.-B 23 (8) (2018)
3347-3360,

\bibitem{Soner11} H. Soner, N. Touzi, J. Zhang,  Martingale representation theorem under $G$-expectation, Stochastic
Process. Appl., 121 (2011) 265-287.

\bibitem{Song11}  Y. Song,  Some properties on $G$-evaluation and its applications
to G-martingale decomposition, Sci. China  Math., 54 (2011) 287-300.

\bibitem{Song12}  Y. Song, Uniqueness of the representation for $G$-martingales
with finite variation, Electron. J. Probab., 17 (2012) 1-15.

\bibitem{Sucec87} J. Sucec, Practical stability analysis of finite difference equations
by the matrix method, Int. J. Numer. Methods Eng. 24 (1987) 679-687.

\bibitem{RF17} R. Ullah and F. Faizullah,  On existence and approximate solutions for stochastic differential
equations in the framework of $G$-Brownian motion, Eur. Phys. J.
Plus,  132 (2017): 435.

\bibitem{LiYang19} Q. Yang,  G. Li, Exponential stability of $\theta$-method for
stochastic differential equations in the $G$-framework. J. Comput.
Appl. Math. 350 (2019) 195-211.

\bibitem{YLWW20} Q. Yao, P. Lin, L. Wang, Y. Wang, Practical exponential stability of impulsive stochastic reaction-diffusion systems with delays, IEEE
Trans. Cybern. 52 (2020) 1-11.

\bibitem{YZ20} Z. Yao, X. Zong, Delay-dependent stability of a class of stochastic delay systems driven by $G$-Brownian motion,
IET Control Theory \& Applications 14 (2020) 834-842.

\bibitem{Yuan06} C. Yuan, W. Glover, Approximate solutions of stochastic differential delay equations
with Markovian switching, J. Comput. Appl. Math. 194 (2006) 207-226.

\bibitem{Yhy21} H. Yuan, Some Properties of Numerical Solutions for Semilinear Stochastic
Delay Differential Equations Driven by $G$-Brownian Motion, Math.
Probl. Eng. 30 (2021) 1-26.

\bibitem{YuanZhu24} H. Yuan, Q. Zhu, Practical stability of the analytical and numerical
solutions of stochastic delay differential equations driven by
G-Brownian motion via some novel techniques,  Chaos Soliton. Fract.
183 (2024) 114920.

\bibitem{ZhCh12} D. Zhang, Z. Chen, Exponential stability for
stochastic differential equation driven by $G$-Brownian motion,
Appl. Math. Lett. 25 (2012) 1906-1910.

\bibitem{zhao08} P. Zhao, Practical stability, controllability and optimal control of
stochastic Markovian jump systems with time-delays, Automatica 44
(2008) 3120-3125.

\bibitem{zhsongLiu23} Y. Zhang, M. Song, M. Liu, Equivalence of stability among stochastic differential
equations, stochastic differential delay equations, and their
corresponding Euler-Maruyama methods, Discrete Contin. Dyn. Syst.-B
28 (2023) 4761-4779.

\bibitem{Zhu22} D. Zhu, Practical exponential stability of stochastic
delayed systems with $G$-Brownian motion via vector $G$-Lyapunov
function. Math. Comput. Simulat. 199 (2022) 307-316.

\bibitem{ZhuHuang20} Q. Zhu, T. Huang, Stability analysis for a class of
stochastic delay nonlinear systems driven by G-Brownian motion,
Syst. Control Lett. 140 (2020) 104699.

\bibitem{ZYL22} D. Zhu, J. Yang, X. Liu, Practical stability of impulsive
stochastic delayed systems driven by $G$-Brownian motion. J.
Franklin I. 359 (2022) 3749-3767.
\end{thebibliography}
\end{document}